\documentclass[11pt,a4paper]{amsart}
\usepackage[latin1]{inputenc}   
\usepackage{ dsfont }
\usepackage{amsfonts,amsmath,layout}

\newcommand{\Z}{\mathbb Z}

\newcommand{\R}{\mathbb R}

\def\QED{\hfill $\; \Box$\medskip}
\def\qed{\QED}

\def\R{\mathbb R}

\def\Z{\mathbb Z}

\newcommand{\be}{\begin{equation}}
\newcommand{\ee}{\end{equation}}

\def\1{{\bf 1}}
\def\rife#1{(\ref{#1})}
\def\m{\noalign{\medskip}}
\def\vfi{\varphi}

\def\parint{\int_0^T\!\!\into}
\def\into{\int_{\Omega}}

\def\dive{{\rm div}}
\def\de{\delta}

\def\vep{\varepsilon}
\def\elle#1{L^{#1}(\Omega)}
\def\parelle#1{L^{#1}(Q_T)}

\def\limitate#1{L^\infty(0,T; \elle{#1})}
\def\continue#1{C^0([0,T]; \elle#1)}

\def\me{m^\vep}

\def\tm{\tilde m}
\def\tu{\tilde u}
\def\tmd{\tilde m_\de}
\def\md{m_\de}

\def\tb{\tilde b}
\def\tw{\tilde w}
\def\ds{\displaystyle}
\numberwithin{equation}{section}

\newtheorem{Theorem}{Theorem}[section]
\newtheorem{definition}[Theorem]{Definition}
\newtheorem{prop}[Theorem]{Proposition}
\newtheorem{lemma}[Theorem]{Lemma}

\newtheorem{remark}[Theorem]{Remark}

\begin{document}

\title{Mean field games with congestion}
\author[Yves Achdou]{Yves Achdou}
\author[Alessio Porretta]{Alessio Porretta}

\address{Yves Achdou
\hfill \break\indent Univ. Paris Diderot, Sorbonne Paris Cit{\'e}, Laboratoire Jacques-Louis Lions, UMR 7598, UPMC, CNRS, F-75205 Paris, France , 
}{}
\email{achdou@ljll-univ-paris-diderot.fr}

 \address{Alessio Porretta
\hfill \break\indent Dipartimento di Matematica, Universit{\`a} di Roma ``Tor Vergata'', Via della Ricerca Scientifica 1,
00133 Roma, Italia}{}
\email{porretta@mat.uniroma2.it}

\thanks{}

\maketitle

\begin{abstract}
  We consider a class of systems of time dependent partial differential equations which arise in mean field type models with congestion. The systems couple a  backward viscous 
 Hamilton-Jacobi equation and a forward Kolmogorov equation both posed in $(0,T)\times (\R^N /\Z^N)$. 
  Because of congestion and by contrast with simpler cases,  the latter system can never be seen as the optimality conditions of an optimal control problem  driven by a  partial differential equation.
 The Hamiltonian vanishes as the density tends to $+\infty$ and may not even be defined in the regions where the density is zero.
After giving a suitable definition of  weak solutions,  we prove the existence and uniqueness results of the latter under rather general assumptions.
 No restriction is made on the horizon $T$.
\end{abstract}
\section{Introduction}
\label{sec:introduction}
Recently, an important research activity on mean field games (MFGs for short) has been initiated since the pioneering works \cite{LL06cr1,LL06cr2,LL07} of  Lasry and Lions: 
it aims at studying the  asymptotic behavior of  stochastic differential games (Nash equilibria) as the number $n$ of agents
tends to infinity.  In  these models, it is assumed that the agents are all identical and that  an individual agent can hardly influence the outcome of the game. 
 Moreover, each individual strategy is influenced by  some averages of functions  of 
the states of the other agents. In the limit when $n\to +\infty$, a given agent feels the presence of the others  through the 
statistical distribution of the states. Since perturbations of a single agent's strategy does not influence the statistical states distribution, 
the latter acts as a parameter  in the control problem to be solved by each  agent.  When the dynamics of the agents are independent stochastic processes, MFGs
naturally lead to a coupled system of two partial differential equations (PDEs for short), a forward Kolmogorov or Fokker-Planck equation  and a backward Hamilton-Jacobi-Bellman equation, see for example (\ref{MFG0}) below.\\

The theory of MFGs allows one to model congestion effects, i.e.  situations in  which the cost of displacement of the agents increases
 in those regions where the density is large. MFGs models including congestion   were introduced and studied in \cite{L-coll}.  
A typical such model  leads to the following system of PDEs:
\begin{equation}\label{MFG0}
\begin{cases}
 -\partial_t u - \nu \Delta u +\frac1\beta\frac{|Du|^\beta}{(m+\mu)^\alpha}  = F(t,x,m)\,,
  & \quad (t,x)\in (0,T)\times \Omega
  \\
  \partial_t m  -\nu \Delta m -{\rm div} (m \frac{  |Du|^{\beta-2}Du}{(m+\mu)^\alpha})=0\,, &\quad  (t,x)\in (0,T)\times \Omega
  \\
  m(0,x)=m_0(x)\,,\,\,u(T,x)= G(x,m(T))\,, & \quad x\in \Omega\,,
\end{cases} 
\end{equation}
with $\nu>0$, $\alpha>0$, $\beta\in (1,2]$, $\mu\in \R$ with either $\mu>0$ or $\mu=0$. System (\ref{MFG0}) must generally be complemented with suitable boundary conditions on $(0,T)\times \partial \Omega$, but we will avoid the additional technical difficulties coming from the latter by  focusing on the case when $\Omega$ is the flat torus, i.e.  
$\Omega={\mathbb T}^N=\R^N /\mathbb Z^N$  and all the data are periodic.
\\
Loosely speaking, (\ref{MFG0}) describes the optimization over a stochastic dynamics defined on  a standard probability space $({\mathcal X}, {\mathcal F}, {\mathcal F}_t, {\mathbb P})$
$$
dX_t= w_t\,dt + \sqrt {2\nu}\, dB_t
$$
where  $B_t$ is a  $N$-dimensional Brownian motion, 
with a cost criterion  given by  
\be\label{Lag}
\inf_{\beta}\,\,\left[{\mathbb E} \int_0^T \left\{c_\beta\,\,(m_t+\mu)^\gamma\, |w_t|^{\frac\beta{\beta-1}}+ F(t,X_t, m_t)\right\}dt+ {\mathbb E}( G(X_T, m_T))\right]\,,
\ee
where $\gamma= \frac\alpha{\beta-1}$ and $c_\beta$ is a  suitable normalization constant.
 In the viewpoint of the generic agent, $m_t=m(t,X_t)$ is meant to represent  the distribution law of the states, however in the optimization process it is just, a  priori, a given frozen density function. The mean field game equilibrium is next given  
through a  fixed point scheme, by requiring,  a  posteriori,   that $m_t$ coincides with  the probability density in $\R^N$ of  the law of the optimal process $X_t$. \\

In \cite{L-coll}, P-L. Lions put the stress on general stuctural conditions yielding the uniqueness for the following  MFG systems  with local coupling:
\begin{eqnarray}
  \label{MFGu} -\partial_t u - \nu \Delta u + H(t,x, m,Du)  = F(t,x,m)\,,
  & \quad (t,x)\in (0,T)\times \Omega
  \\   \label{MFGm}
  \partial_t m  -\nu \Delta m -{\rm div} (m H_p(t,x,m,Du))=0\,, &\quad  (t,x)\in (0,T)\times \Omega
  \\   \label{MFGbc}
  m(0,x)=m_0(x)\,,\,\,u(T,x)= G(x,m(T))\,, & \quad x\in \Omega\,  
\end{eqnarray}
namely  that $F$ and $G$ be increasing  w.r.t. $m$ and that the following matrix be positive semidefinite:
\begin{equation}\label{hyp-uniq}
  \begin{pmatrix}
 -\frac  2  m \frac {\partial  H} {\partial m}    \left(t,x,m,p\right) &    \frac {\partial } {\partial m}  \nabla_p ^T H(t,x,m,p)
\\
  \frac {\partial } {\partial m}  \nabla_p  H(t,x,m,p) &  2D^2_{p,p}  H(t,x,m,p)
 \end{pmatrix}
 \geq 0
\end{equation}
for all $x\in \Omega$, $m>0$ and $p\in \R^N$. 
Since (\ref{MFG0}) is equivalent to (\ref{MFGu})-(\ref{MFGbc}) with  $H(t,x,m,p)=\frac1\beta\frac{|p|^\beta}{(m+\mu)^\alpha}$,
 (\ref{hyp-uniq}) becomes in this case
\begin{equation}
  \label{eq:50}
\alpha\leq \frac{4(\beta-1)}\beta.
\end{equation}
In the present work, we will show that this hypothesis yields both the existence and the uniqueness of weak solutions.
\\
Except for situations in which  special tricks may be applied (stationary problems and quadratic Hamiltonian, see  \cite{Gomes-cong}),  the existence of classical solutions of  suitable generalizations of (\ref{MFG0}) seems difficult to obtain, because generally neither upper bounds on $m$  nor strict positivity of $m$ are known unless one restricts the growth conditions for the nonlinearities and assumes the time horizon $T$ to be small, see \cite{GV15}, \cite{Graber15} (see also  \cite{Eva-Gomes} for the stationary case).
Therefore, in order to get at  a sufficiently general result, we  aim at proving the existence and uniqueness of suitably defined weak solutions.

For MFG models without congestion, the first results on the existence of weak solutions were  supplied in \cite{LL07}.
Besides,  as already observed in \cite{LL06cr1,LL06cr2,LL07}, in the easiest cases, the system of PDEs can be seen as the optimality conditions 
of  a  problem of optimal control driven by a PDE: in such cases,  a pair of primal-dual optimization problems can be introduced,
 leading  to a suitable weak formulation   for which there exists a unique solution, see \cite{CGPT},
where possibly degenerate diffusions are dealt with. A striking fact is that in general,  MFGs with congestion cannot be cast into an optimal control problem 
driven by a PDE, by contrast with simpler cases. For MFG systems \rife{MFGu}--\rife{MFGbc}  with  $H$  independent of $m$, a complete analysis is available in \cite{Parma}, which contains in particular
new results on weak solutions of Fokker-Planck equations,  and an answer to the delicate question of the uniqueness of weak solutions of MFG systems.
Proving uniqueness of weak solutions is difficult for the following reason: to compare two solutions $(u,m)$ and $(\tilde u, \tilde m )$ of (\ref{MFG0}),
 the main idea is to test the Bellman equations by $m-\tilde m$ and the Kolmogorov equations by $u-\tilde u$ and  sum the resulting identities.
While this is of course permitted for classical solutions, special care is needed for weak solutions, because the PDEs only hold in the distributional sense.
The present work will borrow several ideas and results from \cite{Parma}. 

Let us   mention that other kinds of models, which may also include congestion,  are obtained by  assuming that all the agents use the same distributed feedback strategy
 and by passing  to the limit as $N\to \infty$ before optimizing the common feedback. Given a common feedback strategy, the asymptotics are 
given by the McKean-Vlasov theory, \cite{MR0221595}: the dynamics of a given agent is found by solving  a stochastic differential equation with coefficients depending on 
a mean field, namely the statistical distribution of the states, which may also affect the objective function. Since the feedback strategy is common to all agents, perturbations of the latter affect the mean field.  Then, letting each player optimize its objective function amounts to solving a control problem driven by the McKean-Vlasov dynamics. The latter is named control of McKean-Vlasov dynamics by R. Carmona and F. Delarue \cite{MR3045029,MR3091726} and mean field type control by A. Bensoussan et al, \cite{MR3037035,MR3134900}.  By contrast with MFGs, this  genuine  connection to optimal control problems driven by a  PDE makes  it possible to use techniques from the calculus of variations and define a suitable notion of weak solutions for which existence and uniqueness can be proved, see \cite{AL}.  
\\
Finally, let us mention that numerical methods and simulations for MFGs 
and mean field type control with congestion are dealt with in \cite{MR3135339,YAML}.
  \\
The present paper is devoted to prove the existence and uniqueness of weak solutions to a class of  MFG systems which generalize \rife{MFG0} including the congestion effects in  the structure conditions of the Hamiltonian function.  Making reference to the model example \rife{MFG0}, we will consider  both the case when $\mu>0$ and the case when  $\mu=0$.  Let us notice that the congestion effect is essentially contained in the behavior of the Hamiltonian for {\it large values of $m$}, so both cases could be considered as congestion models, since the difference between $\mu>0$ and $\mu=0$
will be important in those regions of vanishing density. Hereafter, the case that $\mu=0$ will be referred to as the case of {\it singular congestion}, since in this case the Hamiltonian may  not even be  defined at $m=0$. In the latter case,   the Lagrangian vanishes at $m=0$, and this vanishing of the  cost criterion will lead to a  slightly incomplete information on the equation solved by  the value function, 
and therefore to  a  relaxed formulation of the notion of weak solution.

The work is organized as follows: in Section 2  we set the problem and introduce the assumptions that will be used throughout the paper.  In Section 3 we  define the notions of weak solutions and state the main existence and uniqueness results: as explained above, we will give a different definition of solution in the case $\mu=0$,
in which  special care is required because the Hamiltonian has a singularity  at $m=0$.  In Section 4, we consider non singular Hamiltonians ($\mu>0$) and give the first steps of the existence proof which consist of studying a sequence of suitably regularized problems.
Section 5 contain  lemmas which will turn crucial for proving uniqueness but also for identifying the limit of the aforementioned regularized problems.
In Section 6,  in the case $\mu>0$, we conclude the proof of existence and obtain uniqueness  under suitable additional assumptions.
Section 7 is devoted to the singular case $\mu=0$.

\section{Running assumptions}
\label{sec:running-assumptions-1}

Let us consider the general MFG system (\ref{MFGu})-(\ref{MFGbc}). Different type of boundary conditions could be considered but  for simplicity we neglect this point by assuming that the equation takes place in a standard flat torus $\Omega={\mathbb T}^N=\R^N /\mathbb Z^N$ and all functions are $\mathbb Z^N$-periodic in $x$.  
We define $Q_T=(0,T)\times \Omega$.

 In (\ref{MFGu})-(\ref{MFGbc}), the function   $H(t,x,m,p)$ is assumed to be  measurable with respect  to $(t,x)$, continuous with respect to $m$ and $C^1$ with respect to $p$.
The notation $H_p$ stands for $\frac{\partial H(t,x,m,p)}{\partial p}$. We also assume that $H$  is  a convex function of $p$.


We modify the structure growth conditions  introduced in  \cite{LL07} to take into account the congestion factor. So, we will assume that $H$ satisfies, for some positive $\beta$ and $\alpha$
\begin{eqnarray}
H(t,x,m,0)\le 0\,, &\label{eq:1} \\
H(t,x,m,p) \geq c_0\,  \frac{|p|^\beta}{(m+\mu)^\alpha}- c_1 \,(1+m^{\frac\alpha{\beta-1}})\,, &  \label{coerc}
\\
|H_p(t,x,m,p)| \leq c_2\,  (1+ \frac{|p|^{\beta-1}}{(m+\mu)^\alpha})\,,  & 
\label{grow}
\\
H_p(t,x,m,p)\cdot p  \geq (1+\sigma)\, H(t,x,m,p)  - \, c_3\,(1+m^{\frac\alpha{\beta-1}})\,,
& \label{conv}
\end{eqnarray}
for a.e.  $(t,x)\in Q_T$ and every $(m,p)\in \R_+\times  \R^N$,  where $\sigma,c_0,\ldots,c_3$ are positive constants and $\mu\in \R$, $\mu\geq 0$. For the sake of clarity, we will separately consider  the two cases  $\mu>0$ and $\mu=0$.

Notice that \rife{eq:1} can be alternatively rephrased as: $H(t,x,m,0)$ is bounded from above; indeed, if $
H(t,x,m,0)\le  K$ for some $K>0$, we can reduce to \rife{eq:1} by changing $u$ into $u+K(T-t)$.

\bigskip

As for the ranges of $\beta,\alpha$ in the previous assumptions, we will assume that
\begin{equation}\label{alfabeta}
1<\beta \leq 2\,,\qquad 0<\alpha\leq \frac{4(\beta-1)}\beta\,.
\end{equation}
The first condition on $\beta$ is meant to exclude both the case of linear or sub linear growth and the case of super quadratic growth. 
Actually, if $\beta\leq 1$, then the bound on  $H_p$ would imply that $m$ is bounded, and this would allow for a  huge simplification in the system (at least if $\mu>0$), in particular the difficult issues coming from the congestion effect would be mostly circumvented and this case would be rather similar to the results which already exist in the literature. 
The super quadratic case is different in nature since one cannot rely any more on classical tools for parabolic equations and  a different approach should be employed, as in the first order case. Since this would bring us too far for the present paper, we decided to assume that  $\beta\leq 2$.

As for (\ref{alfabeta}),  this structural condition seems  natural in order to have a completely well-posed formulation in  the congestion model. 
As already mentioned, this limitation  is indeed required in order to obtain uniqueness, as pointed out by P.L. Lions \cite{L-coll}, when he discussed general conditions for the uniqueness of mean field games systems.
In addition, this condition also seems to be useful at other stages of our existence proof. The upper bound required on $\alpha$ can be interpreted 
 as a tolerable  growth condition on  the cost function which penalizes the motion in the zones where the distribution density is high (since we have $\gamma= \frac \alpha{\beta-1}$ in the Lagrangian model \rife{Lag}).

Finally, the case $\alpha=0$ is excluded in the above conditions since this is already treated in \cite{Parma}.  
\begin{remark}\label{sec:running-assumptions-2}
  Straightforward calculus leads to  $4/\beta ' \le \beta$: from this and  (\ref{alfabeta}), we see that $\alpha\le \beta$, and that $\alpha=\beta$ can occur only if $\beta=\alpha=2$.
\end{remark}
Notice that combining (\ref{eq:1}) and the convexity of $H$ in the $p$ variable yields that
\begin{equation}
  \label{eq:2}
H_p(t,x,m,p)\cdot p -H (t,x,m,p)\ge -H(t,x,m,0)\ge 0. 
\end{equation}
Moreover,  since $1<\beta\le 2$, and from (\ref{coerc})-(\ref{conv}), we can deduce that  there exist   nonnegative constants $C_0, C_1, C_2$ such that
\begin{equation}\label{hp2}
 m^{\frac\alpha{\beta-1}+1 } \left( |H_p|^2 - C_0\right) \leq C_1 m \left\{H_p\cdot p- H + C_2\right\} .
\end{equation}
Using generic non negative constants $C_0, C_1$ that can change from line to line, (\ref{hp2}) is obtained as follows:
\begin{displaymath}
  \begin{split}
    m^{\frac\alpha{\beta-1}}  |H_p|^2 &\le C_1  \left( m^{\frac\alpha{\beta-1}}+ \frac {|p|^{2(\beta-1)}} {(m+\mu )^{2\alpha -\frac\alpha{\beta-1}}}\right)\\
 &\le  C_1 m^{\frac\alpha{\beta-1}}+    c_0 \frac {|p|^{\beta}}{(m+\mu )^\alpha} 
+ C_0 (m+\mu)^{\frac \alpha {\beta-1}}
\\
&\le  H+  C_1 (m^{\frac\alpha{\beta-1}} +1 ) 
\\
&\le \sigma^{-1}(H_p\cdot p -H)  + C_1 + C_0 m^{\frac\alpha{\beta-1}}
  \end{split}
\end{displaymath}
where we have used (\ref{grow}) in the first line,  $1< \beta\le 2$ and Young's inequality in the second line, (\ref{coerc}) in the third line and (\ref{conv}) in the fourth line.

We also stress that, if $H$ satisfies \rife{eq:1}--\rife{conv}, then $H+ b\cdot p$ satisfies the same conditions for any bounded vector field $b$. Hence the addition of bounded drift terms in the dynamics is permitted 
by the above setting of assumptions. 

\begin{remark}
Let us  compare the above assumptions with previous settings used for congestion models in mean field games. In \cite{GV15}, 
  congestion models are presented starting  from the Lagrangian function
$$
L(t,x,m,q)=m^\gamma L_0(t,x,q-b(t,x)) +F(t,x,m),\quad \hbox{with } \quad  L_0(t,x,q)= a(t,x) (1+ |q|^2 )^{\frac {\beta '} 2},
$$
where  $\beta'=  \beta /(\beta -1)$,  $a$ is a smooth positive function and $b$ is a smooth velocity field. This leads to the system 
$$
\begin{cases}
 -\partial_t u - \nu \Delta u +m^{\gamma} H(t,x,\frac{Du}{m^\gamma})  = F(t,x,m)\,,
  & \quad (t,x)\in (0,T)\times \Omega
  \\
  \partial_t m  -\nu \Delta m -{\rm div} (m H_p(t,x,\frac{Du}{m^\gamma}))=0\,, &\quad  (t,x)\in (0,T)\times \Omega
  \\
  m(0,x)=m_0(x)\,,\,\,u(T,x)= G(x,m(T))\,, & \quad x\in \Omega\,
  \end{cases}
$$
where
\begin{displaymath}
  H(t,x,p)= H_0(t,x,p)+ b(t,x)\cdot p, \quad \hbox{and }\quad    H_0(t,x,p)= \sup_{q} -q\cdot p - L_0(t,x,q).
\end{displaymath}
If we set $\alpha= \gamma(\beta-1)$, we recover our assumptions if $\gamma \le 4/\beta$ which is equivalent to saying that
$\nabla_p H_0 (t,x,p)\cdot p- H_0(t,x,p)\ge \frac \gamma  4   p^T D_{pp}H_0 (t,x,p) p$ for all $p\not = 0$, as assumed in \cite{GV15}.
\end{remark}
\vskip0.4em

Finally, let us  make precise the assumptions on the functions $F,G$ appearing in \rife{MFGu}--\rife{MFGbc}.
The function $F$ is assumed to be measurable with respect to $(t,x)\in Q_T$ 
and continuous with respect to $m$.  
In addition, we assume that   there exists   a nondecreasing  function  $f(s)$ such that $f(s)s$  is convex and
\begin{equation}\label{F1}
  \lambda\, f(m)- \kappa \leq F(t,x,m)\leq  \frac1\lambda\,f(m)+\kappa \quad \forall m\in \R_+\,,
\end{equation}
for some real numbers $\lambda, \kappa>0$.

The  terminal cost $G$ is assumed to be measurable with respect to $x\in \Omega$, continuous with respect to $m$ and such that
\be\label{G1}
m\mapsto G(x,m)\qquad \hbox{is nondecreasing.}
\ee
In addition, as for $F$, we assume that  there exists a  function $g(s)$ such that $g(s)s$  is convex and
\begin{equation}\label{G2}
  \lambda\, g(m)- \kappa \leq G(x,m)\leq  \frac1\lambda\,g(m)+\kappa \quad \forall m\in \R_+\,, 
\end{equation}
with for example the same constants  $\lambda, \kappa>0$ as in \rife{F1}.

Let us notice that, as a consequence of assumptions \rife{F1} and \rife{G2}, the functions $F, G$ are bounded below, namely there exists a constant $c_4\in \R$ such that
\begin{equation}\label{F0}
F(t,x,m) \geq c_4\,,\quad  \forall \,\, m\in \R_+\,,\,\, {\rm a.e.} \,\, (t,x)\in Q_T\,,
\end{equation}
and
\be\label{G0}
  G(x,m)\geq c_4\,,\quad  \forall \,\, m\in \R_+\,,\,\, {\rm a.e.} \,\, x\in \Omega\,.
\end{equation}

Assumptions (\ref{eq:1})--\rife{alfabeta}   and (\ref{F1})--(\ref{G2}) are the structural conditions under which we will prove the existence of weak solutions to the MFG system with congestion. Eventually, we will  need a  further condition for $H$ and $F$ in order to have uniqueness. This will be  a natural reformulation of condition \rife{hyp-uniq}. Precisely, we assume the following: 
\be\label{eq:44}
\begin{split}
& \hbox{for any $m\ge 0$, $z\ge -m$,  $p,r \in \R^N$ such that $(z,r)\not=(0,0)$, setting}
\\ & \quad 
\hbox{  $m_s=m+sz$ and  $p_s= p+sr$ for $s\in [0,1]$, then the function}
\\
\m &  \qquad  \qquad 
h: s\mapsto -z H(t,x,m_s, p_s) + m_s H_{p}(t,x,m_s, p_s) \cdot r +  z F(t,x, m_s)
\\ 
\m
& 
\hbox{is (strictly) increasing on $[0,1]$.} 
\end{split}
\end{equation}
\vskip0.3em
\begin{remark}\label{sec:case-non-singular-4}
  Note that if $H(t,x,m,p)= (\mu+m)^{-\alpha} |p|^\beta$, with $(\alpha, \beta)$ satisfying (\ref{alfabeta})
 and if $F$ is  nondecreasing with respect to $m$, then $h$ defined in (\ref{eq:44}) is increasing.  Condition (\ref{alfabeta}) is thus compatible 
with the assumption yielding uniqueness.
\end{remark}

\vskip1em
Finally,  we assume $m_0\in  C(\Omega)$, $m_0\geq 0$ and, in order to be consistent with the interpretation  of $m$, we fix the normalization condition $\into m_0=1$.  

\section{Main results}
\label{sec:main-results}

\subsection{Case of non singular congestion} 
\label{sec:case-non-singular-2}
\quad 
We assume here that conditions \rife{eq:1}--\rife{conv} hold with $\mu>0$.
Let us first make our notion of weak solution precise.

\begin{definition}\label{weaksol} A pair $(u,m)\in \parelle1\times \parelle1 $ is a weak solution  to (\ref{MFGu})-(\ref{MFGbc})  if 
\vskip0.3em
(i) $m\in C([0,T];\elle1)$, $m\geq 0$, $u\in L^q(0,T; W^{1,q}(\Omega))$ for every $q<\frac{N+2}{N+1}$,  
\vskip0.3em
(ii) 
\begin{align*}
& F(t,x,m)m\in \parelle1\,,\qquad G(x,m(T))m(T)\in \elle1\,,
\\
 &   m\,\frac{|Du|^\beta}{(m+\mu)^\alpha}\in \parelle1\,, \qquad  \frac{|Du|^\beta}{(m+\mu)^\alpha}\in \parelle1\,, \qquad m\in \parelle{1+\frac\alpha{\beta-1}}\,
\end{align*}
\vskip0.3em
(iii)  $u\in \limitate1$ is bounded below  and is   a solution of the Bellman equation in the sense of distributions:
\begin{equation}
  \label{eq:17}
  \begin{split}
\int_0^T \into u\, \vfi_t\, dxdt & - \nu\int_0^T \into u\,\Delta\vfi\, dxdt 
+ \int_0^T\into H(t,x,m,Du)\vfi\, dxdt
\\
& = \int_0^T \into F(t,x,m)\vfi\, dxdt + \into G(x,m(T))\vfi(T)\, dx     
  \end{split}
\end{equation}
for every $\vfi\in C^\infty_c((0,T]\times \Omega)$.

(iv) $m$ is a solution of the Kolmogorov equation:
\begin{equation}
  \label{eq:16}
\int_0^T\into m\left\{ -\vfi_t-\nu\Delta \vfi+H_p(t,x,m,Du)D\vfi\right\}\, dxdt =\into m_0\,\vfi(0)\, dx
\end{equation}
for every $\vfi\in C^\infty_c([0,T)\times \Omega)$.
\end{definition} 

\begin{remark}\label{ureno}
We notice that, on account of conditions (ii), \rife{eq:17} implies that $u\in \parelle1$   solves, in weak sense, 
$$
\begin{cases}
 -\partial_t u - \nu \Delta u =f(t,x)\,,
  & \quad (t,x)\in (0,T)\times \Omega
  \\
  u(T,x)= g(x)\,, & \quad x\in \Omega\,
  \end{cases}
$$
for some $f\in \parelle1$, $g\in \elle1$. It follows that $u\in \continue1$ and, in addition, $u$ is a  renormalized solution of the same equation (see e.g. \cite[Proposition 2.1]{Parma}). 
Indeed, the condition that $u$  belongs to $\limitate1$ is actually redundant in (iii), since it  follows from the above result.
\end{remark}

As stated in the previous remark,  a solution $u$ of equation \rife{eq:17} belongs to $\continue1$. However we also need to work with merely {\it subsolutions} of the same equation, for which this kind of continuity may not hold.  We recall a lemma, whose proof can be found in \cite{CGPT,AL}, which establishes continuity for subsolutions in a weaker sense.

\begin{lemma}\label{sec:case-non-singular}
Let  $u\in L^\infty(0,T;\elle1)$ satisfy, for some function $h\in L^1(Q_T)$ and $k\in L^1(\Omega)$,
\begin{displaymath}
\int_0^T \into u\, \vfi_t\, dxdt  - \nu \int_0^T \into u\,\Delta\vfi\, dxdt  \le  \int_0^T \into h \vfi\, dxdt + \into k(x)\vfi(T,x)\, dx,
\end{displaymath}
for every nonnegative function $\vfi\in C^\infty_c((0,T]\times \Omega)$.  

Then for any Lipschitz continuous map $\xi: \Omega \to \R$, the map $t\mapsto \int_{\Omega} \xi(x) u(t, x) dx$
 has a BV representative on $[0,T]$. Moreover, if we note $\int_{\Omega} \xi(x) u(t^+, x) dx$ its right limit at $t\in [0,T)$,
 then the map $\xi \mapsto \int_{\Omega} \xi(x)u(t^+, x) dx$  can be extended to a  bounded linear form on  $C(\Omega)$.
\end{lemma}
\begin{remark}
\label{sec:case-non-singular-1}
As a consequence of Lemma~\ref{sec:case-non-singular}, for any subsolution  $u$ of \rife{eq:17} we can define 
$u(0^+)$ as a bounded Radon measure on $\Omega$. For simplicity, we note $u(0)=u(0^+)$.
  \end{remark}

  \begin{remark}
    \label{sec:case-non-singular-3}
    A weaker definition than  Definition~\ref{weaksol} is possible, by replacing (ii) by  $G(x,m(T))\in \elle1$, 
    $F(t,x,m) \in \parelle1$, $H(t,x,m,Du)\in \parelle1$ and $m\,H_p(t,x,m,Du)\in \parelle1$. In this case,  we would need to derive firstly the regularity (ii) in order to prove uniqueness and this could be itself a delicate step. It does not seem restrictive to ask directly (ii) in the formulation since such regularity is actually obtained in the existence result.
 \end{remark}
 
 Finally, let us state below our main result, where we prove existence and uniqueness of weak solutions under the above assumptions.
   
\begin{Theorem} \label{sec:case-non-singular-5}
Under Assumptions (\ref{eq:1})--\rife{conv} with $\mu>0$, (\ref{alfabeta}) and (\ref{F1})--(\ref{G2}), 
 there exists a weak solution of problem (\ref{MFGu})-(\ref{MFGbc}).  \\
Furthermore, if the assumption  \rife{eq:44} is satisfied, 
then there is a unique weak solution of problem (\ref{MFGu})-(\ref{MFGbc}). 
\end{Theorem}

\subsection{Case of singular congestion}
\label{sec:case-sing-cong}

Here we consider the limit case $\mu=0$ in assumptions \rife{coerc}--\rife{grow}, which corresponds to the singular congestion model $H= \frac{|Du|^\beta}{m^\alpha}$. Due to the singularity, we will only find a relaxed solution. \\
Before defining weak solutions, we  notice that    
we can always change $H(t,x,m,p)$ into $H(t,x,m,p)+c_1$ and $u$ into $u-c_1(T-t)$; therefore, in view of (\ref{coerc}), we may  assume without restriction that
\begin{equation}
  \label{eq:41}
H(t,x,0,p)\ge 0.
\end{equation}
This allows us to define weak solutions as follows:
\begin{definition}\label{defsing} A pair $(u,m)\in \parelle1\times \parelle1_+$ is a weak solution  to (\ref{MFGu})-(\ref{MFGbc})  if
\vskip0.3em
(i) $m\in C([0,T];\elle1)$, $u\in L^q(0,T; W^{1,q}(\Omega))$ for every $q<\frac{N+2}{N+1}$,
\vskip0.3em
(ii) 
\begin{equation}\label{mDu}
\begin{split}
&F(t,x,m)m\in \parelle1\,, \qquad G(x,m(T))m(T)\in \elle1\,,\\
&   m\, \mathds{1}_{m>0} \,\frac{|Du|^\beta}{m^\alpha}\in \parelle1\,, \qquad   \mathds{1}_{m>0} \frac{|Du|^\beta}{m^\alpha}\in \parelle1\,, 
\qquad m\in \parelle{1+\frac\alpha{\beta-1}},\,
\\
&  Du=0 \quad \hbox{a.e. in $\{m=0\}$,}
\end{split}
\end{equation}
(iii) $u\in \limitate1$ is bounded below  and is a subsolution of the Bellman equation:
\begin{equation}
  \label{eq:38}
  \begin{split}
    \int_0^T \into u\, \vfi_t\, dxdt & - \nu\int_0^T \into u\,\Delta\vfi\, dxdt 
+ \int_0^T\into H(t,x,m,Du)\mathds{1}_{\{m>0\}}\vfi\, dxdt
\\
& \leq  \int_0^T \into F(t,x,m)\vfi\, dxdt + \into G(x,m(T))\vfi(T)\, dx 
  \end{split}
\end{equation}
 for every nonnegative $\vfi\in C^\infty_c((0,T]\times \Omega)$.

(iv) $m$ is a solution of the Kolmogorov equation:
\begin{equation}
  \label{eq:37}
\int_0^T\into m\left\{ -\vfi_t-\nu\Delta \vfi+H_p(t,x,m,Du)\mathds{1}_{\{m>0\}}\, D\vfi\right\}\, dxdt =\into m_0\,\vfi(0)\, dx
\end{equation}
for every $\vfi\in C^\infty_c([0,T)\times \Omega)$.

(v) The following identity is satisfied:
\begin{equation}\label{en-id}
\begin{split}
 \into   m_0\, u(0)\, dx & = \into G(x,m(T))\,   m(T)\,dx + \parint F(t,x,m)  m\, dxdt
\\ & + \parint m\,  \left[   H_p(t,x,  m, D  u)\cdot Du-  H(t,x,m,Du)\right] \mathds{1}_{\{m>0\}}  dxdt
\end{split}
\end{equation}
where the first term is understood as the trace of $\into u(t)\, m_0\, dx$ in $BV(0,T)$, in view of Lemma \ref{sec:case-non-singular}.
\end{definition} 

Notice that, if $\alpha\leq 1$,  we actually have $m^{1-\alpha} |Du|^\beta\in \parelle1$ and the restriction to the set $\{m>0\}$ in \rife{eq:37}--\rife{en-id} is not needed. However, for the case $\alpha>1$ this restriction applies.
\vskip0.3em
We state below the  existence and uniqueness result which we obtain for the case of singular congestion. 
\begin{Theorem}\label{sec:case-sing-cong-1}
Assume that  (\ref{eq:1})--\rife{conv} hold with $\mu=0$ and with either $\beta<2 $ and $0< \alpha\le   \frac {4(\beta-1)}\beta$ or  $\beta=2 $ and $0< \alpha<2$, and  that  (\ref{F1})--(\ref{G2}) hold true (assuming furthermore the nonrestrictive condition~(\ref{eq:41})). Then there exists a  weak solution of problem (\ref{MFGu})-(\ref{MFGbc}).\\
Furthermore, if  condition \rife{eq:44} holds true  
   for any $m> 0$, $z> -m$,   $p,r \in \R^N$ such that $(z,r)\not=(0,0)$,  and if 
 \begin{equation}
   \label{eq:49}
m>0\Rightarrow F(t,x,m)>F(t,x,0),
 \end{equation}
then there is a unique weak solution of problem (\ref{MFGu})-(\ref{MFGbc}). 
\end{Theorem}

\section{ Non singular congestion: approximations of~(\ref{MFGu})-(\ref{MFGbc})}
\label{sec:non-sing-cong-6}
For $\mu>0$, we consider the following system of PDEs:
\begin{eqnarray}
\label{eq:3u}
 -\partial_t u^\epsilon - \nu \Delta u^\epsilon +
 H(t,x, T_{1 /\epsilon} m^\epsilon,Du^\epsilon)  = F^\epsilon(t,x,m^\epsilon)\,,
  & \quad (t,x)\in (0,T)\times \Omega
  \\\label{eq:3m}
  \partial_t m^\epsilon  -\nu \Delta m^\epsilon -{\rm div} (m^\epsilon H_p(t,x,T_{1/ \epsilon} m^\epsilon,Du^\epsilon))=0\,, &\quad  (t,x)\in (0,T)\times \Omega
  \\\label{eq:3bc}
  m^\epsilon(0,x)=m_0^\epsilon(x)\,,\,\,u^\epsilon(T,x)=  G^\epsilon(x,m^\epsilon(T))\,, & \quad x\in \Omega\, 
\end{eqnarray}
where $T_{1/\epsilon}  m=  \min ( m,1/\epsilon)$, 
 $ F^\epsilon(t,x,m)=  \rho^\epsilon \star F( t,\cdot, \rho^\epsilon \star m) )   (x)$, 
$G^\epsilon(x,m)=  \rho^\epsilon \star G( \cdot, \rho^\epsilon \star m) )   (x)$,  $m_0^\epsilon= \rho^\epsilon \star m_{0}$. Here 
$\star$ denotes the convolution with respect to the spatial variable 
 and $\rho^\epsilon$ is a standard symmetric mollifier, i.e. $\rho^\epsilon(x)= \frac 1 {\epsilon^N} \rho(\frac x \epsilon)$ for  a  nonnegative function $\rho\in C_c^\infty(\R^N)$ such that
$\int_{\R^N} \rho(x)dx=1$.
\begin{lemma}
  \label{sec:non-sing-cong}
There exists a weak solution $(u^\epsilon,m^\epsilon)$ of  (\ref{eq:3u})-(\ref{eq:3bc}) such that 
$u^\epsilon\in C^{1+\alpha, 1/2+\alpha/2}(\bar Q_T)$, $m^\epsilon\in C^{\alpha, \alpha/2}(\bar Q_T)$, and $m^\epsilon\ge 0$, $ \int_\Omega m^\epsilon(t,x) dx=\int_\Omega m_0^\epsilon(x) dx$,  $\forall t$.
\end{lemma}
\begin{proof} 
 We set  $X=\left\{m\in C(\bar Q_T):  m\ge 0 \hbox{ and }  \int_\Omega m(t,x) dx= \int_\Omega m_0(x) dx,\quad \forall t\right\}$.
For any $m\in X$, the boundary value problem
\begin{eqnarray}
\label{sec:non-sing-cong-7} -\partial_t u- \nu \Delta u +
 H(t,x, T_{1 /\epsilon} m,Du)  = F^\epsilon(t,x,m)\,,
  & \quad (t,x)\in (0,T)\times \Omega
  \\
\label{sec:non-sing-cong-9}  u(T,x)=  G^\epsilon(x,m(T))\,, & \quad x\in \Omega\, 
\end{eqnarray}
has a unique solution $u\in  L^\infty(0,T; W^{1,\infty}(\Omega))$ (see e.g. \cite{LSU}, Chapter 5, Theorem 6.3) and $\|u\|_{L^\infty(0,T; W^{1,\infty}(\Omega))}$ is bounded independently of $m$ in $X$. This follows e.g. from the $C^{1,\alpha}$ estimates in \cite{LSU}, Chapter 5, Theorem 3.1, since $F^\epsilon$ is bounded and $ |H(t,x, T_{1 /\epsilon} m,Du) | \leq C_\epsilon (1+ |Du|^2)$ for some constant $C_\epsilon>0$.
In addition, the map $m\mapsto u$ is continuous from $X$ to $ L^\infty(0,T; W^{1,\infty}(\Omega))$.
\\
Thus, the boundary value problem
\begin{eqnarray*}
    \partial_t \tilde m  -\nu \Delta \tilde m -{\rm div} (\tilde m  H_p(t,x,T_{1/ \epsilon} m^\epsilon,Du))=0\,, &\quad  (t,x)\in (0,T)\times \Omega
  \\
  \tilde m(0,x)=m_0^\epsilon(x)\,, & \quad x\in \Omega
\end{eqnarray*}
has a unique solution $\tilde m\in X\cap C^{\alpha, \alpha/2}(\bar Q_T)$, which is bounded in $ C^{\alpha, \alpha/2}(\bar Q_T)$ uniformly with respect to $m\in X$, see e.g. \cite{LSU}, Chapter 3, Theorem 10.1. Moreover the map $m\mapsto \tilde m$ is continuous from $ X$ to $C^{\alpha, \alpha/2}(\bar Q_T)$.
\\
Hence, Schauder's  theorem implies that the map $m\mapsto \tilde m$ has a fixed point $m^\epsilon$.  
\end{proof}

\begin{lemma}
\label{sec:non-sing-cong-1}
Let $(u^\epsilon, m^\epsilon)$ be a solution  of system  (\ref{eq:3u})-(\ref{eq:3bc}).  Then
\begin{eqnarray}
\label{eq:21}u^\epsilon (t,x)\ge c_4 ,\\
  \label{eq:11}
\into G^\epsilon(x,m^\epsilon(T))m^\epsilon(T) +  \parint F^\epsilon(t,x,m^\epsilon)m^\epsilon\, dxdt  
 + \|(T_{1/\epsilon} m^\epsilon)^{\frac\alpha{\beta-1}+1}\|_{\parelle{\frac{N+2}N}}
\leq C,
\\
\label{eq:12}
 \parint m^\epsilon \left\{H_p(t,x, T_{1/\epsilon} m^\epsilon,Du^\epsilon)\cdot Du^\epsilon- H(t,x,T_{1/\epsilon} m^\epsilon,Du^\epsilon)\right\}\,dxdt \leq C,
\\\label{est-mDu}
\parint \frac{|Du^\epsilon|^\beta}{( T_{1/\epsilon} m^\epsilon+\mu)^\alpha} \, dxdt +
\parint m^\epsilon\, \frac{|Du^\epsilon|^\beta}{( T_{1/\epsilon} m^\epsilon+\mu)^\alpha} \, dxdt   \leq C,
\\\label{eq:32}
\|u^\epsilon\|_{L^\infty(0,T; L^1(\Omega))}\le C, 
\end{eqnarray}
for a positive constant  $C =C(T, H, F, G,  \|m_0\|_\infty)$ and where $c_4$ is the constant appearing in (\ref{F0}) and (\ref{G0}).
\end{lemma}

\proof 
We multiply the equation of $u^\epsilon$ by $m^\epsilon$, the equation of $m^\epsilon$ by $u^\epsilon$, integrate by parts some terms and
subtract; we get
\begin{align*}& 
\into u^\epsilon(0)m_0^\epsilon\, dx - \into G^\epsilon(x,m^\epsilon(T))m^\epsilon(T) 
- \parint m^\epsilon H_p(t,x, T_{1/\epsilon}m^\epsilon,Du^\epsilon)\cdot Du^\epsilon\,dxdt
\\
& \quad + \parint m^\epsilon\, H(t,x, T_{1/\epsilon}m^\epsilon,Du^\epsilon)\,dxdt = \parint F^\epsilon(t,x,m^\epsilon)m^\epsilon\, dxdt,
\end{align*}
which implies that
\begin{equation}\label{est1}
\begin{split}
& 
\into G^\epsilon(x,m^\epsilon(T))m^\epsilon(T) +  \parint F^\epsilon(t,x,m^\epsilon)m^\epsilon\, dxdt  +\into u^\epsilon(0) ^- m_0 \epsilon\, dx
\\
& +   \parint m^\epsilon \left\{H_p(t,x,T_{1/\epsilon} m^\epsilon,Du^\epsilon)\cdot Du^\epsilon- H(t,x,T_{1/\epsilon} m^\epsilon,Du^\epsilon)\right\}\,dxdt 
\\
\le \, & C\, \| m_0\|_{\elle\infty} \into [u^\epsilon(0)]^+\, dx\,.
\end{split}
\end{equation}
By using  (\ref{eq:1}) and $F,G\geq c_4$, we see by comparison that
$u^\epsilon(t)\ge  c_4 + c_4(T-t)$,  and therefore that there exists an absolute constant $C$ 
such that 
\begin{equation}
  \label{eq:5}
[u^\epsilon(0)]^- \le C.
\end{equation}
Integrating (\ref{eq:3u}) and using (\ref{coerc}) and (\ref{eq:5}), we obtain that 
\begin{equation}\label{est2}
  \begin{split}
&\into [u^\epsilon(0)]^+\, dx + c_0 \parint \frac{|Du^\epsilon|^\beta}{( T_{1/\epsilon} m^\epsilon+\mu)^\alpha} \, dxdt\\ \leq & c_1 \parint (1+ (T_{1/\epsilon} m^\epsilon) ^{\frac\alpha{\beta-1}})dxdt + \parint F^\epsilon(t,x,m^\epsilon)dxdt +  \into G^\epsilon(x,m^\epsilon(T))dx     +C 
\\
= &  c_1 \parint (1+ (T_{1/\epsilon} m^\epsilon) ^{\frac\alpha{\beta-1}})dxdt + \parint F(t,x,\rho^\epsilon \star m^\epsilon)dxdt +
  \into G(x, \rho^\epsilon\star m^\epsilon(T))dx     +C \\
  \end{split}
\end{equation}
{From} Assumption  (\ref{F1}) we see that,   for any choice of $L>0$,
\be\label{magf} 
F(t,x,m)\leq  f_{L}(t,x)+ \frac m {L}   (F(t,x,m) + |c_4|), 
\ee
where $c_4$ is the constant in \rife{F0} and $ f_{L}(t,x)=\max_{0\le m\le L} F(t,x,m)\le  \frac 1 \lambda \max_{0\le m\le L} f(m) + \kappa$, from (\ref{F1}).

Similarly, $G(x,m)= \leq  g _L(x)+ \frac m L  ( G(x,m) + |c_4|)$  where $g_L(x)=\max_{0\le m\le L} G(x,m)\le \frac 1 \lambda \max_{0\le m\le L} g(m) 
+ \kappa$, from~(\ref{G2}).
Therefore, (\ref{est2}) implies
\begin{equation}\label{eq:6}
  \begin{split}
& \into [u^\epsilon(0)]^+\, dx  + c_0 \parint \frac{|Du^\epsilon|^\beta}{( T_{1/\epsilon} m^\epsilon+\mu)^\alpha} \, dxdt 
\leq   c_1 \parint (1+ (T_{1/\epsilon} m^\epsilon) ^{\frac\alpha{\beta-1}})dxdt \\ & \qquad +
\frac 1 L \left(\parint m^\epsilon  F^\epsilon(t,x,m^\epsilon)dxdt +  \into m^\epsilon(T)G^\epsilon(x,m^\epsilon(T))dx\right)     +C_L, 
\end{split}
\end{equation}
where $C_L$ is a constant depending on $L$. Therefore, by choosing $L$ sufficiently large (depending on $\|m_0\|_{\elle\infty}$)
 we can combine \rife{est1} and (\ref{eq:6}) and obtain 
\begin{equation}\label{est3}
\begin{split}& 
\into G^\epsilon(x,m^\epsilon(T))m^\epsilon(T) +  \parint F^\epsilon(t,x,m^\epsilon)m^\epsilon\, dxdt  + \parint \frac{|Du^\epsilon|^\beta}{( T_{1/\epsilon} m^\epsilon+\mu)^\alpha} \, dxdt
\\
& +   \parint m^\epsilon \left\{H_p(t,x, T_{1/\epsilon} m^\epsilon,Du^\epsilon)\cdot Du^\epsilon- H(t,x,T_{1/\epsilon} m^\epsilon,Du^\epsilon)\right\}\,dxdt 
\\ \leq \, &\, C + C \parint  ( T_{1/\epsilon} m^\epsilon )^{\frac\alpha{\beta-1}} dxdt\,.
\end{split}
\end{equation}
We now use (\ref{eq:3m}) multiplied by $( T_{1/\epsilon} m^\epsilon )^{\frac\alpha{\beta-1}} $. We get
\begin{displaymath}
  \begin{split}
  &\iint \partial_t m^\epsilon  ( T_{1/\epsilon} m^\epsilon )^{\frac\alpha{\beta-1}}  \, dxdt + 
 \frac {\nu \alpha}{\beta-1}     \iint_{m^\epsilon < 1/\epsilon}  |Dm^\epsilon |^2 (m^\epsilon)^{\frac\alpha{\beta-1}-1}\, dxdt
\\ =&
- \frac\alpha{\beta-1}\iint_{m^\epsilon < 1/\epsilon}   (m^\epsilon)^{\frac\alpha{\beta-1}}
 H_p(t,x,  m^\epsilon,Du^\epsilon) \cdot Dm^\epsilon   \, dxdt\,.    
  \end{split}
\end{displaymath}
Calling $I_\epsilon(z)=  \int_0 ^z  ( T_{1/\epsilon} y )^{\frac\alpha{\beta-1}}  dy$, we see that 
\[ \ds \int_0^t  \int_\Omega \partial_t m^\epsilon  ( T_{1/\epsilon} m^\epsilon )^{\frac\alpha{\beta-1}}  \, dxds  =     \int_\Omega
 I_\epsilon( m^\epsilon (t) )  dx    -    \int_\Omega
 I_\epsilon( m_0^\epsilon )  dx .  \]
Therefore,  by Young's inequality
\begin{equation}
  \label{eq:4}
  \begin{split}
& \sup_{t}    \int_\Omega
 I_\epsilon( m^\epsilon (t) )  dx    -    \int_\Omega
 I_\epsilon( m_0^\epsilon )  dx
+ \iint_{m^\epsilon < 1/\epsilon}  |Dm^\epsilon|^2 (m^\epsilon)^{\frac\alpha{\beta-1}-1}\, dxdt\\
\leq  \,&\,  C \iint_{m^\epsilon < 1/\epsilon}  (m^\epsilon)^{\frac\alpha{\beta-1}+1} |H_p(t,x,m^\epsilon,Du^\epsilon)|^2\, dxdt\,.    
  \end{split}
\end{equation}
Since $I_\epsilon(z)\ge    \frac 1 {\frac\alpha{\beta-1} +1} ( T_{1/\epsilon} z )^{\frac\alpha{\beta-1}+1}$ we use (\ref{hp2}) to obtain 
\begin{equation}
  \label{eq:8}
  \begin{split}
&      \sup_{t\in (0,T)}     \int_\Omega   (T_{1/\epsilon} m^\epsilon(t) )^{\frac\alpha{\beta-1} +1}   \, dx
+ \parint  \left|D  T_{1/\epsilon} m^\epsilon\right|^2 (T_{1/\epsilon} m^\epsilon)^{\frac\alpha{\beta-1}-1}\, dxdt\\
\leq \, &\,   C  \iint_{m^\epsilon < 1/\epsilon} m^\epsilon \left\{ H_p(t,x,m^\epsilon,Du^\epsilon)\cdot Du^\epsilon- H(t,x,m^\epsilon,Du^\epsilon) \right\}dxdt\\ & +
C \left[1+  \iint_{m^\epsilon < 1/\epsilon} (m^\epsilon)^{\frac\alpha{\beta-1}+1}\, dxdt\right] + C \int_\Omega   (m_0^\epsilon)^{\frac\alpha{\beta-1} +1} dx 
\\
\leq\, &\,  C \parint m^\epsilon \left\{ H_p(t,x,T_{1/\epsilon} m^\epsilon,Du^\epsilon)\cdot Du^\epsilon- H(t,x,T_{1/\epsilon} m^\epsilon,Du^\epsilon) \right\}dxdt
\\ &+C \left[1+  \parint  ( T_{1/\epsilon}  m^\epsilon)^{\frac\alpha{\beta-1}+1}\, dxdt\right] + C \int_\Omega   (m_0^\epsilon)^{\frac\alpha{\beta-1} +1} dx ,
\end{split}
\end{equation}
where the last inequality is a consequence of (\ref{eq:2}).
On the other hand, the left-hand side of (\ref{eq:8}) can be estimated by  Gagliardo-Nirenberg  interpolation inequality, which implies
\begin{displaymath}
  \begin{split}
&\|   (T_{1/\epsilon} m^\epsilon)^{\frac\alpha{\beta-1}+1}\|_{\parelle{\frac{N+2}N}} \\ \leq \, &\,  C\,  \left[\sup_{t} \into (T_{1/\epsilon} m^\epsilon (t))
^{\frac\alpha{\beta-1}+1}\,dx +
 \parint |DT_{1/\epsilon} m^\epsilon |^2  (T_{1/\epsilon} m^\epsilon)^{\frac\alpha{\beta-1}-1}\, dxdt\right]\,.      
  \end{split}
\end{displaymath}
Then, when comparing with the right-hand side of (\ref{eq:8}), and recalling that $m^\epsilon$ is already bounded in $L^1$, we conclude that
\begin{displaymath}
  \begin{split}
&\|(T_{1/\epsilon} m^\epsilon)^{\frac\alpha{\beta-1}+1}\|_{\parelle{\frac{N+2}N}}\\ \leq\, & 
\, C \parint m^\epsilon \left\{ H_p(t,x,T_{1/\epsilon} m^\epsilon,Du)\cdot Du^\epsilon- H(t,x,T_{1/\epsilon} m^\epsilon,Du^\epsilon) \right\}dxdt +
C \,,      
  \end{split}
\end{displaymath}
where $C$ also depends on $m_0$. We use now this information in \rife{est3} and we deduce that
\begin{displaymath}
\begin{split}& 
\into G^\epsilon(x,m^\epsilon(T))m^\epsilon(T) +  \parint F^\epsilon(t,x,m^\epsilon)m^\epsilon\, dxdt  
 + \|(T_{1/\epsilon} m^\epsilon)^{\frac\alpha{\beta-1}+1}\|_{\parelle{\frac{N+2}N}}
\\\leq \, &\, C + C \parint  ( T_{1/\epsilon} m^\epsilon )^{\frac\alpha{\beta-1}} dxdt\,.
\end{split}
\end{displaymath}
The latter inequality implies that its right hand side is bounded uniformly in $\epsilon$, and we obtain (\ref{eq:11}).
  From \rife{est3},  we also obtain (\ref{eq:12}) and we estimate the first term in (\ref{est-mDu}).
  
Finally, from (\ref{eq:4}), we see that $\sup_{t} \int_\Omega I_\epsilon( m^\epsilon (t) )  dx\le C$ which implies that
\begin{equation}
  \label{eq:7}
\sup_{t} \int_\Omega  m^\epsilon(t) (T_{1/\epsilon}m^\epsilon(t))^{\frac\alpha{\beta-1}} \,dx\le C.
\end{equation}
Then using (\ref{coerc}) and (\ref{conv}), (\ref{eq:12}) and (\ref{eq:7}) imply that 
\begin{equation}
  \label{eq:10}  \parint m^\epsilon \frac {  |Du^\epsilon|^\beta }{ (\mu + T_{1/\epsilon} m^\epsilon) ^\alpha} dxdt\le C,
\end{equation}
which completes (\ref{est-mDu}). In particular, we also deduce that  $  \parint m^\epsilon \frac {  |Du^\epsilon|^\beta }{ (\mu +  m^\epsilon) ^\alpha} dxdt\le C$.
\\
Finally, integrating (\ref{eq:3u}) in $(t,T)\times \Omega$, we get that
\begin{displaymath}
  \begin{split}
&  \int_\Omega u^\epsilon(t,x)dx\\=  &\int_\Omega G( x, \rho^\epsilon \star m^\epsilon(T) ) dx +\int_t^T\int_\Omega F (t,x,\rho^\epsilon \star m^\epsilon(s)) dxds - \int_t^T\int_\Omega  H(s,x,T_{1/ \epsilon}  m^\epsilon, Du^\epsilon) dx ds .    
  \end{split}
\end{displaymath}
The last term is bounded above by $C$ from (\ref{eq:11}) and (\ref{est-mDu}). 
Let us deal with $\int_\Omega G( x, \rho^\epsilon \star m^\epsilon(T) ) dx $:  from (\ref{G0}),  $G( x, \rho^\epsilon \star m^\epsilon(T) )
 \le g_L(x) + \frac 1 L \left( \rho^\epsilon \star m^\epsilon(T) G( x, \rho^\epsilon \star m^\epsilon(T) ) +|c_4| \rho^\epsilon \star m^\epsilon(T)  \right)$,
where $L$ is an arbitrary positive number and $g_L(x)=\max _{0\le m\le L } G(x,m)\le \frac 1 \lambda \max_{0\le m\le L} g(m) 
+ \kappa$, from~(\ref{G2}).
Noting that $\int_\Omega \rho^\epsilon \star m^\epsilon(T) G( x, \rho^\epsilon \star m^\epsilon(T) ) dx= \int_\Omega  m^\epsilon G^\epsilon(x, m^\epsilon) dx$,
and that $\int_\Omega m^\epsilon(x)= \int_\Omega m_0(x)$, we deduce  from  (\ref{eq:11}) that  $\int_\Omega G( x, \rho^\epsilon \star m^\epsilon(T) ) dx \le C$, for
some constant $C$ independent of $\epsilon$ and $\mu$.
\\
The same argument can be used for proving that $\int_t^T\int_\Omega F (t,x,\rho^\epsilon \star m^\epsilon(s)) dxds\le C$. Thus 
\begin{displaymath}
   \int_\Omega u^\epsilon(t,x)dx \le C.
\end{displaymath}
Combining this with the lower bound on $u^\epsilon$ already obtained, we obtain (\ref{eq:32}).
%
\qed 
\begin{remark}
  Note that the constant $C$ appearing in Lemma~\ref{sec:non-sing-cong-1} does not depend on $\mu$.
\end{remark}

\begin{lemma}
  \label{sec:non-sing-cong-2}
For each $0<\epsilon<1$, there exist two functions $w^\epsilon$ and $z^\epsilon$ defined on $Q_T$ such that 
$$
\left|m^\epsilon H_p(t,x,T_{1/\epsilon}m^\epsilon, Du^\epsilon)\right| \le c_2 m^\epsilon+  w^\epsilon + \sqrt {m ^\epsilon } z^\epsilon\,,
$$ 
where the family $(w^\epsilon)$ is bounded in $L^{\beta '} (Q_T)$, the family $(z^\epsilon )$ is bounded in  $L^2(Q_T)$ and, in addition,  relatively compact if $\beta<2$.
\end{lemma}
\begin{proof}
From Remark \ref{sec:running-assumptions-2}, we know that $\alpha \leq \beta$.
 Therefore, if $\alpha\ge 1$, then $\alpha' \geq  \beta'$. This yields 
 \begin{equation}
   \label{eq:9}
(1-\alpha) \beta' \geq  -\alpha, \quad \quad \hbox{for any} \,\,\, \alpha\,  \hbox{ such that }  0< \alpha \le \frac {4(\beta-1)}{\beta},
 \end{equation}
because (\ref{eq:9}) is trivial if $0<\alpha\le 1$.
\\
From (\ref{grow}), we see that $\left|m^\epsilon H_p(t,x,T_{1/\epsilon}m^\epsilon, Du^\epsilon)\right| \leq c_2\, m^\epsilon\,
  (1+ \frac{|Du^\epsilon|^{\beta-1}}{(T_{1/\epsilon}m^\epsilon+\mu)^\alpha})  \le  c_2 (m^\epsilon + A_1+A_2)$, where
$A_1=  \mathds{1}_{\{T_{1/\epsilon} m^\epsilon +\mu \le 1   \}} \, \frac {|Du^\epsilon|^{\beta-1}}{(T_{1/\epsilon} m^\epsilon + \mu)^{\alpha-1} } $
and 
$A_2=  \mathds{1}_{\{T_{1/\epsilon} m^\epsilon +\mu \ge 1   \}} \,m^\epsilon\, \frac {|Du^\epsilon|^{\beta-1}}{ (T_{1/\epsilon}m^\epsilon + \mu)^{\alpha}} $.
\\
Let us first deal with $A_1$:  (\ref{eq:9}) implies that $A_1^{\beta '} \le    \mathds{1}_{\{T_{1/\epsilon}m^\epsilon +\mu \le 1   \}} \, \frac {|Du^\epsilon|^{\beta}}
{(T_{1/\epsilon} m^\epsilon + \mu)^{\alpha} } \le    \frac {|Du^\epsilon|^{\beta}}
{(T_{1/\epsilon}m^\epsilon + \mu)^{\alpha} }$. Therefore, from (\ref{est-mDu}),  $w^\epsilon=c_2\, A_1$  is bounded in $L^{\beta'}(Q_T)$ uniformly with respect to $\epsilon$
(and to $\mu$).
\\
Next, we see that  $A_2 \le \sqrt{m^\epsilon}  z^\epsilon$ where $z^\epsilon=    \mathds{1}_{\{T_{1/\epsilon} m^\epsilon +\mu \ge 1   \}}  \sqrt{m^\epsilon} 
 \frac {|Du^\epsilon|^{\beta-1}}{(T_{1/\epsilon}m^\epsilon + \mu)^{\alpha} } $. Then, by H{\"o}lder inequality, for any $E\subseteq Q_T$ we have
\be\label{zeps2}
  \begin{split}
    \iint_E( z^\epsilon)^2    dxdt& =     \iint_E  \mathds{1}_{\{T_{1/\epsilon}m^\epsilon +\mu \ge 1   \}}  m^\epsilon
 \frac {|Du^\epsilon|^{2\beta-2}}{(T_{1/\epsilon}m^\epsilon + \mu)^{2\alpha} }   dxdt \\ &\le 
  \left(   \iint_E   m^\epsilon  \frac {|Du^\epsilon|^{\beta}}{  (T_{1/\epsilon}m^\epsilon +\mu) ^{\alpha} } dxdt \right)^{\frac 2 {\beta'}} 
\left(   \iint_E \mathds{1}_{\{T_{1/\epsilon}m^\epsilon +\mu \ge 1   \}}  \frac {m^\epsilon} {(T_{1/\epsilon}m^\epsilon +\mu) ^{\alpha   \left(  1+ \frac {\beta '}{\beta'-2} \right)}}  dxdt \right)^{1-\frac 2 {\beta'}}
\\ &\le 
  \left(  \iint_E   m^\epsilon  \frac {|Du^\epsilon|^{\beta}}{  (T_{1/\epsilon}m^\epsilon +\mu) ^{\alpha} } dxdt \right)^{\frac 2 {\beta'}} 
\left(  \iint_E  m^\epsilon  dxdt \right)^{1-\frac 2 {\beta'}} .
  \end{split}
\ee
We immediately deduce from \rife{est-mDu} that $z^\epsilon$ is bounded in $\parelle2$. In addition, for any $k<\frac1\epsilon$ we have 
$$
\iint_{E}  m^\epsilon  dxdt \leq k\, |E| + \frac1{k^{\frac\alpha{\beta-1}}}\iint_E m^\epsilon\, (T_{1/\epsilon} m^\epsilon )^{\frac\alpha{\beta-1}} \leq k\, |E| + \frac C{k^{\frac\alpha{\beta-1}}}
$$
where we used \rife{eq:7}. This implies that for $\epsilon$ small enough,
$$
\iint_{E}  m^\epsilon  dxdt \leq C\, |E|^{\frac\alpha{\alpha+\beta-1}}
$$
and so  we additionally deduce from \rife{zeps2} that $z^\epsilon$ is equi-integrable in $\parelle2$ provided $\beta<2$.
 \end{proof}

We now collect some compactness properties for the family $(u^\epsilon, m^\epsilon)$, which  are mostly borrowed from  \cite{Parma}. 

 \begin{prop}\label{sec:non-sing-cong-5}
 \ 
 
   \begin{enumerate}
   \item    The family $(m^\epsilon)$ is relatively compact in $L^1(Q_T)$  and in $C( [0,T]; W^{-1, r}(\Omega) )$ for some  $r>1$.
   
   \item There exist $m\in C([0,T]; L^1(\Omega))$ and $u\in \limitate1 \cap L^q(0,T; W^{1,q}(\Omega))$ for all $q<\frac{N+2}{N+1}$ such that  after the extraction of a subsequence (not relabeled), $m^\epsilon\to m$ in $ L^1(Q_T)$ and almost everywhere, $u^\epsilon\to u$ and $Du^\epsilon\to Du$ in $ L^1(Q_T)$ and almost everywhere. Moreover, $u$ and $m$ satisfy  
\begin{eqnarray}
\label{eq:18}
  \parint \frac{|Du|^\beta}{(  m+\mu)^\alpha} \, dxdt +
\parint m\, \frac{|Du|^\beta}{(  m+\mu)^\alpha} \, dxdt   \leq C,\\
  \label{eq:20}
 \into G(x,m(T))m(T)\, dx + \parint mF(t,x,m) \,dxdt   +    \| m^{\frac\alpha{\beta-1}+1}\|_{\parelle{\frac{N+2}N}}\leq C,
\end{eqnarray}
for some $C=C(T, H, F, G,  \|m_0\|_\infty)$ (independent of $\mu$).
 
 \item  $m^\epsilon(t)\to m(t)$ weakly in $\elle1$, for every $t\in [0,T]$, and (\ref{eq:16}) holds for any $\varphi\in C^\infty_c([0,T)\times  \Omega)$. If in addition $\beta<2$, then   $m^\epsilon\to m$ in $C([0,T]; L^1(\Omega))$.
\end{enumerate}
 \end{prop}

 \begin{proof}
Proceeding as in the end of the proof of Lemma \ref{sec:non-sing-cong-1}, we see that $(F^\epsilon(t,x,m^\epsilon))_\epsilon $ is bounded in $L^1(Q_T)$,
$(u^\epsilon(t=T))_\epsilon$ is bounded in $L^1  (\Omega)$ and that $H(t,x, T_{1/\epsilon}m, Du^\epsilon) $ is bounded in $L^1 (Q_T)$ uniformly in $\epsilon$.
Therefore, $\left(-\partial_t u^\epsilon -\nu \Delta u^\epsilon\right)_\epsilon $ is bounded in $L^1(Q_T)$  and 
$(u^\epsilon(t=T))_\epsilon$ is bounded in $L^1  (\Omega)$. The compactness of $u^\epsilon$ and $Du^\epsilon$ in $\parelle1$ (and almost everywhere) follow from classical 
results on the heat equation with $L^1$ data, see e.g. \cite{BDGO}, \cite{BG}: there exist $u\in \limitate1 \cap L^q((0,T);W^{1,q}(\Omega)) $ for every $q< \frac {N+2}{N+1}$, such that   after the extraction of a subsequence
$u^\epsilon\to u$ and $Du^\epsilon\to Du$ in $ L^1(Q_T)$ and almost everywhere. 
 
As far as $m^\epsilon$ is concerned, thanks  to  Lemma~\ref{sec:non-sing-cong-2} we can apply the estimates and the compactness results in   \cite[Theorem 6.1]{Parma}. In particular, this implies the  point (1) in the statement and so, up to extraction of  a subsequence, the convergence of $m^\epsilon $ in $\parelle1$ and almost everywhere.
\\
The bounds (\ref{eq:18}) and (\ref{eq:20}) follow from (\ref{eq:11}), (\ref{est-mDu}) and Fatou's lemma, keeping in mind that the constant $C$ in the  estimates of Lemma~\ref{sec:non-sing-cong-1} does not depend on $\mu$.
\\
To prove that $m$ is a solution of (\ref{eq:16}),   we observe that  $m^\epsilon H_p(t,x,T_{1/\epsilon}m^\epsilon, Du^\epsilon)$ strongly converges in $\parelle1$. Indeed, the decomposition provided by   Lemma \ref{sec:non-sing-cong-2} implies that this term is equi-integrable, so using the almost everywhere convergence and Vitali's theorem, we obtain that $m^\epsilon H_p(t,x,T_{1/\epsilon}m^\epsilon, Du^\epsilon)\to  m H_p(t,x,m, Du)$ in $L^1 (Q_T)$.
As a consequence, (\ref{eq:16}) holds for any $\varphi\in C^\infty_c([0,T)\times \Omega)$.\\
Finally,  we have that $m^\epsilon(t)$ is bounded in $L\log L(\Omega)$ uniformly in time (see e.g. \cite{Parma}), which means that it is equi-integrable in $\elle1$. By Dunford-Pettis theorem, it is weakly relatively compact in $\elle1$, and since we already know that it converges to $m(t)$ in  $W^{-1, r}(\Omega)$, we conclude that it also converges weakly in $\elle1$ to the same limit. In fact, if $ \beta<2$, from Lemma \ref{sec:non-sing-cong-2}   we have the strong convergence of the terms $w^\epsilon, z^\epsilon$ in $\parelle2$. This means that we can apply directly  Theorem 6.1 in \cite{Parma} to deduce the strong convergence of $m^\epsilon$ in  $C([0,T]; L^1(\Omega))$.
\end{proof}

Henceforth, we let $(u,m)$ as well as the convergent subsequence $(u^\epsilon, m^\epsilon)$  be given by Proposition \ref{sec:non-sing-cong-5}. Let us notice that the weak convergence of $m^\epsilon(T)$ in $\elle1$ easily yields that 
$ \rho^\epsilon \star m^\epsilon(T)$ also converges to $ m(T) $ weakly in $\elle1$. In particular, it is well-known that the sequence $\rho^\epsilon \star m^\epsilon(T)$ generates a family of parametrized Young measures $\{\nu_x\}\in {\mathcal P}(\R)$, see e.g. \cite{Ball}, \cite{Ped}. This means that  $\nu_x$ is a probability measure for a.e. $x\in \Omega$, the mapping $x \mapsto \nu_x$ is weakly-$*$ measurable and, for a subsequence (not relabeled), 
\be\label{proyou}
f(x,  \rho^\epsilon \star m^\epsilon(T)) \rightharpoonup \int_{\R} f(x,\lambda) d\nu_x(\lambda)
\qquad \hbox{weakly in $\elle1$}
\ee 
for every Carath{\'e}odory function $f(x,s)$ such that $f(x,  \rho^\epsilon \star m^\epsilon(T))$ is equi-integrable in $\elle1$.

Thanks to the Young measures, we can initially identify the limit of $u^\epsilon(T)$ and give a first description,   by now in  a relaxed sense,   of the equation satisfied by $u$. 

\begin{lemma}
  \label{sec:non-sing-cong-4}
Let $(u,m)$ be given by Proposition \ref{sec:non-sing-cong-5}. Then 
\begin{equation}
  \label{eq:19}
  \begin{split}
    \int_0^T \into u\, \vfi_t\, dxdt & - \nu \int_0^T \into u\,\Delta\vfi\, dxdt 
+ \int_0^T\into H(t,x,m,Du)\vfi\, dxdt
\\
& \leq  \int_0^T \into F(t,x,m)\vfi\, dxdt + \into \int_{\R} G(x,\lambda)d\nu_x(\lambda)\,\vfi(T)\, dx
  \end{split}
\end{equation}
for every nonnegative function $\vfi\in C^\infty_c((0,T]\times \Omega)$. 

In addition, we have that $\int_{\R} G(x,\lambda)\lambda\, d\nu_x(\lambda)\in \elle1$, $G(x,m(T))m(T)\in \elle1$ and 
\be\label{gmm}
\into  G(x,m(T))m(T)\,dx \leq C
\ee
for some $C=C(T, H, F, G,  \|m_0\|_\infty)$ (independent of $\mu$).
\end{lemma}
\begin{proof}
From (\ref{eq:11}) and the convergence of $m^\epsilon$, we see that $(T_{1/\epsilon} m^\epsilon)^{\frac \alpha  {\beta-1}}$ converges in $L^1(Q_T)$.
Using this observation and (\ref{coerc}), we can  apply Fatou's lemma and obtain that, for any $\varphi\in C^\infty_c((0,T]\times  \Omega)$, $\varphi\ge0$,
\begin{displaymath}
 \parint   H(t,x, m,Du) \varphi(t,x) dxdt \le \liminf_{\epsilon\to 0} \parint   H(t,x, T_{1 /\epsilon} m^\epsilon,Du^\epsilon) \varphi(t,x) dxdt. 
\end{displaymath}
Moreover, from (\ref{eq:11}) and (\ref{F1}), and the definition of $F^\epsilon$, we see that $ F^\epsilon(t,x,m^\epsilon)$ is equi-integrable in $Q_T$. Therefore, from Vitali's theorem,
\begin{equation}
  \label{eq:14}
\lim_{\epsilon\to 0}   \parint   F^\epsilon(t,x,m^\epsilon)\varphi(t,x) dxdt=\parint   F(t,x,m)\varphi(t,x) dxdt.   
\end{equation}
Let us deal with the boundary condition at $t=T$:  again from (\ref{eq:11}), the definition of $G^\epsilon$ and  (\ref{G2}), we deduce that $G(x, \rho^\epsilon\star m^\epsilon(T))$ is equi-integrable, so by the properties of Young measures we obtain that 
\begin{equation}
  \label{eq:15}
  \begin{split}
\lim_{\epsilon\to 0}  \into u^\epsilon(T,x)\varphi(T,x,) dx& =\lim_{\epsilon\to 0}  \into G(x, \rho^\epsilon\star m^\epsilon(T)) \,\rho^\epsilon\star\varphi(T,x) dx 
\\ & =  \into \int_{\R} G(x,\lambda)d\nu_x(\lambda) \varphi(T,x) dx .  
\end{split}
\end{equation}
Combining the latter three points yields (\ref{eq:19}).

Finally, we notice that the bound $\into G^\epsilon(x,m^\epsilon(T))m^\epsilon(T)\,\leq C$ given by \rife{eq:11} implies that both $G(x,m(T))m(T)\in \elle1$ and $\into \int_{\R} G(x,\lambda)\lambda\, d\nu_x(\lambda)\in \elle1$; indeed, on the one hand, thanks to \rife{G2} we have 
$$
\into g(\rho^\epsilon\star m^\epsilon(T))\rho^\epsilon\star m^\epsilon(T)\, dx \leq C
$$
so by convexity of $g(s)s$ and weak convergence of $\rho^\epsilon\star m^\epsilon(T)$ we deduce that $g(m(T))m(T)\in \elle1$, hence $G(x,m(T))m(T)\in \elle1$ again by \rife{G2}, and estimate \rife{gmm} holds.
On the other hand, by properties of Young measures we have 
$$
T_k(G(x, \rho^\epsilon\star m^\epsilon(T))) \rho^\epsilon\star m^\epsilon(T) \rightharpoonup \int_{\R} T_k(G(x,\lambda))\lambda\, d\nu_x(\lambda)
$$
where $T_k(s)= \min(s,k)$. Hence
$$
\into \int_{\R} T_k(G(x,\lambda))\lambda\, d\nu_x(\lambda)\, dx = 
\lim_{\epsilon\to 0} \into T_k(G(x, \rho^\epsilon\star m^\epsilon(T))) \rho^\epsilon\star m^\epsilon(T) \, dx \leq C\,.
$$
By monotone convergence,   letting $k\to \infty$ we deduce that $\int_{\R}  G(x,\lambda)\lambda\, d\nu_x(\lambda)\in \elle1$. 
\end{proof}

\begin{lemma}
  \label{sec:non-sing-cong-8}
Let  $(m^\epsilon, u^\epsilon)$ be given by Proposition  \ref{sec:non-sing-cong-5}. Possibly after another extraction of a subsequence, $u^\epsilon|_{t=0}$ 
converges weakly * to a bounded measure $\chi$  on $\Omega$. Moreover $u(0)$ is a well defined Radon measure on $\Omega$ and  $u(0)\ge \chi$.
\end{lemma}
\begin{proof}
  From (\ref{eq:32}), we know that $u^\epsilon(0)$ is bounded in $L^1(\Omega)$. 
Therefore, there exists some bounded measure $\chi$ such that, up to a subsquence,  $u^\epsilon(0)$ tends to $\chi$ weakly $*$.
 Moreover, from Lemma~\ref{sec:case-non-singular}, Remark~\ref{sec:case-non-singular-1} and Lemma \ref{sec:non-sing-cong-4},   $u(0)$ is a  well defined Radon measure.\\
Let us prove that $\chi \le u(0)$. Consider $\varphi(t,x)= \psi(x) (1- t/\eta)_+ $, where $\psi$ is any nonnegative smooth  function 
defined on $\Omega$.  Taking $\varphi$ as a test-function in (\ref{eq:3u}), we obtain that 
\begin{displaymath}
  \begin{split}
    \int_0^T \into u^\epsilon\, \vfi_t\, dxdt & - \nu \int_0^T \into u^\epsilon\,\Delta\vfi\, dxdt 
+ \int_0^T\into H^\epsilon(t,x,T_{1/\epsilon}m^\epsilon,Du^\epsilon)\vfi\, dxdt
\\
& = \int_0^T \into F^\epsilon(t,x,m^\epsilon)\vfi\, dxdt  - \into \vfi(0,x)u^\epsilon(0,x) dx,
  \end{split}
\end{displaymath}
which implies by using Fatou's lemma as in the proof of Proposition~\ref{sec:non-sing-cong-5} that
\begin{equation}
  \label{eq:3}
  \begin{split}
    \int_0^T \into u\, \vfi_t\, dxdt & - \nu \int_0^T \into u\,\Delta\vfi\, dxdt 
+ \int_0^T\into H(t,x,m,Du)\vfi\, dxdt
\\
& \le \int_0^T \into F(t,x,m)\vfi\, dxdt  -     \langle  \chi , \psi \rangle  .
  \end{split}
\end{equation}
On the other hand,  as $\eta\to 0$,  $  \int_0^T \into u\, \vfi_t\, dxdt=  -\frac 1  \eta  \int_0^{\eta}\int_\Omega u\,\psi(x) \,dx dt$ tends to
 $ -\langle u(0), \psi \rangle$.
Passing to the limit in (\ref{eq:3}) as $\eta\to 0$, we obtain that $  \langle u(0), \psi \rangle \ge  \langle  \chi , \psi \rangle $, which is the desired result.
\end{proof}
 


\section{Crossed energy inequality}
\label{sec:main-comp-lemma}
We prove here the main step towards the uniqueness result.

\begin{lemma}\label{compare}
Assume that $H$ satisfies assumptions \rife{coerc}-\rife{grow} with $\mu>0$, and that \rife{F1}--\rife{G2} hold true.  

Let 
$u\in \limitate1$  satisfy \rife{eq:19} for some family of probability measures $\{\nu_x\}$ (weakly$-*$ measurable w.r.t. $x$) such that  
\be\label{gll}
\int_{\R}G(x,\lambda)\lambda \, d\nu_x(\lambda) \in \elle1
\ee 
and  let $m\in \continue1$ be a solution of  \rife{MFGm}, and assume that $(u,m)$ satisfy (ii)  in Definition~\ref{weaksol}. 
Then we have
\begin{equation}\label{ineq}
\begin{split}
\langle \tilde m_0\,,\, u(0)\rangle & \leq  \into \int_{\R} G(x,\lambda)\, d\nu_x(\lambda) \, \tilde m(T)\,dx + \parint F(t,x,m)\tilde m\, dxdt
\\ & + \parint \left[ \tilde m\, H_p(t,x,\tilde m, D\tilde u)\cdot Du- \tilde m \,H(t,x,m,Du)\right]dxdt
\end{split}
\end{equation}
for any couple $(\tilde u, \tilde m)$ satisfying the same conditions as $(u,m)$.
\end{lemma}

\begin{remark} The above Lemma applies in particular if $\nu_x= \delta_{m(T,x)}$, in which case $u$ is simply a subsolution of \rife{MFGu} and condition \rife{gll} is equivalent to  $G(x,m(T))m(T)\in \elle1$. 
\end{remark}

\proof  
Let $ \rho_\de(\cdot)$ be a sequence of standard  symmetric mollifiers in $\R^N$ and set 
$$
\tilde m_\de(t,x)= \tilde m(t)\star \rho_\de : = \int_{\mathbb R^N} \tilde m(t,y) \rho_\de(x-y)\, dy\,.
$$
Notice in particular that $\tmd, D\tmd \in \parelle\infty$ since $\tm\in \limitate1$. 
We also take a sequence of 1-d mollifiers $\xi_\vep(t)$ such that supp$(\xi_\vep)\subset (-\vep,0)$, and we set 
$$
\tilde m_{\de,\vep}:= \int_0^T \xi_\vep(s-t)\,\tmd(s)ds\,.
$$
Notice that this function vanishes near $t=0$, so we can take it as test function in the inequality satisfied by $u$. We get
\begin{equation}\label{ueq}
\begin{split} & 
 \parint u \,[\partial_t \tilde m_{\de,\vep}-\nu \Delta \tilde m_{\de,\vep}]\, dxdt  + \parint  H(t,x,m,Du)\tilde m_{\de,\vep}\,  dxdt
\\
& \quad \leq \parint  F(t,x,m)\tilde m_{\de,\vep}\,dxdt  + \into \int_{\R} G(x,\lambda)\, d\nu_x(\lambda)\tilde m_{\de,\vep}(T)\, dx 
\end{split}
\end{equation}
The first integral is equal to 
$$
\parint  [-\partial_s u_{\de,\vep} -\nu \Delta u_{\de,\vep}]\tm(s,y) \,  dsdy  
$$
where $u_{\de,\vep}(s,y)= \parint u(t,x)\xi_\vep(s-t)\rho_\de(x-y)\,dtdx$. Notice that this function vanishes near $s=T$. Thus using the equation of $\tm$ we have
\begin{align*}
& \parint  [-\partial_s u_{\de,\vep} -\nu \Delta u_{\de,\vep}]\tm(s,y) \,  dsdy
\\ & \quad = 
-\parint  \tm(s,y)\tb(s,y)\cdot D_y u_{\de,\vep} \,  dsdy + \into \tilde m_0(y) u_{\de,\vep}(0) dy\,,
\end{align*}
where $\tb=H_p(t,x,\tm, D\tu)$.  We shift the convolution kernels from $u$ to $\tilde m$ in the right-hand side and we use this equality in \rife{ueq}. We get   
\begin{equation}\label{prevep}
\begin{split} & 
-\parint   Du \cdot \tw_{\de,\vep} \, dtdx  + \parint  H(t,x,m,Du)\tilde m_{\de,\vep}\,  dxdt + \into  (\tilde m_0\star \rho_\de) \int_0^T u(t)\xi_\vep(-t) \, dt\,\,dx
\\ & \quad  
\leq  \parint  F(t,x,m)\tilde m_{\de,\vep}\,dxdt  + \into   \int_{\R} G(x,\lambda)\, d\nu_x(\lambda)\tilde m_{\de,\vep}(T)\, dx 
\end{split}
\end{equation}
where we denote  $\tw_\de= [(\tb\,\tm)\star \rho_\de]$ and $\tw_{\de,\vep}= \int_0^T \tw_\de(s)\, \xi_\vep(s-t)\,ds$.  
\vskip1em

Since now, we distinguish between the cases $\alpha\leq1 $ and $\alpha>1$. Recall that $\alpha\leq \frac{4(\beta-1)}\beta$ is always assumed, hence the latter case only happens for $1< \frac{4(\beta-1)}\beta$, i.e. $\beta >\frac 43$.
\vskip1em
{\bf (A) Case when $\mathbf \alpha\leq 1$.} \quad 
Let us deal with  the  first two integrals in \rife{prevep}. By assumption \rife{grow} we have
\begin{align*}
|(\tm\, \tb)\star \rho_\de| & \leq c_2 \tmd + c_2 \left(\left(\tm\, \frac{|D\tu |^{\beta}}{(\tm+\mu)^{\alpha}}\right)\star \rho_\de\right)^{\frac1{\beta'}} \left(\frac\tm{(\tm+\mu)^{\alpha}}\star \rho_\de\right)^{\frac1\beta}
\\ &
\leq c_2 \tmd + c_2 \left(\left(\tm\, \frac{|D\tu |^{\beta}}{(\tm+\mu)^{\alpha}}\right)\star \rho_\de\right)^{\frac1{\beta'}}  (\tm\star \rho_\de)^{\frac{1-\alpha}\beta}
\end{align*}
where we used $\alpha\leq 1$ in the latter inequality. Therefore, using also \rife{coerc} we estimate 
\begin{align*}
& H(t,x,m,Du)\tmd(s)- Du \cdot \tw_\de(s) \geq c_0 \tmd(s) \frac{|Du|^{\beta}}{(m+\mu)^{\alpha}}- c_1  \tmd(s) (1+m^{\frac\alpha{\beta-1}}) 
\\
& - c_2 \tmd(s) |Du| -
c_2 |Du|\, \left(\left(\tm(s)\, \frac{|D\tu(s) |^{\beta}}{(\tm(s)+\mu)^{\alpha}}\right)\star \rho_\de\right)^{\frac1{\beta'}}\tmd(s)^{\frac{1-\alpha}\beta}\end{align*}
which yields, by Young's inequality,
\begin{equation}\label{stimainf}
\begin{split}
& H(t,x,m,Du)\tmd(s)- Du \cdot \tw_\de(s) \geq \frac{c_0}2 \tmd(s) \frac{|Du|^{\beta}}{(m+\mu)^{\alpha}}- c\, \tmd(s) (1+m^{\frac\alpha{\beta-1}}) 
\\
&  - \sigma \, |Du|^\beta  \tmd(s)^{1-\alpha} - C_\sigma \,\left(\tm(s)\, \frac{|D\tu(s) |^{\beta}}{(\tm(s)+\mu)^{\alpha}}\right)\star \rho_\de \,.
\end{split}
\end{equation}
Since $\alpha\leq 1$, choosing $\sigma<\frac{c_0}2$   we have  
$$
\sigma \, \tmd^{1-\alpha} < \frac{c_0}2  \frac{\tmd}{(m+\mu)^{\alpha}} + \frac{c_0}2  (m+\mu)^{1-\alpha}\,,
$$
while
$$
\tmd  (1+m^{\frac\alpha{\beta-1}}) \leq c \,(1+ \tmd^{1+\frac{\alpha}{\beta-1}}+   m^{1+\frac{\alpha}{\beta-1}})\,.
$$
So we conclude from \rife{stimainf} that
\begin{align*}
& H(t,x,m,Du)\tmd(s)- Du \cdot \tw_\de(s) \geq  - c\,  (\tmd(s)^{1+\frac{\alpha}{\beta-1}}+m^{1+\frac\alpha{\beta-1}}) 
\\
&  - c\, (m+\mu)^{1-\alpha} \,  |Du|^\beta  - c \,\left(\tm(s)\, \frac{|D\tu(s) |^{\beta}}{(\tm(s)+\mu)^{\alpha}}\right)\star \rho_\de \,.
\end{align*}
Notice that   condition (ii) in Definition \ref{weaksol} implies  that $(m+\mu)^{1-\alpha} |Du|^\beta$   belongs to $\parelle1$. Due to the summability of all the above terms, we are allowed to use Fatou's lemma and deduce
\begin{align*}
\liminf_{\vep\to 0}  &  \parint  [H(t,x,m,Du)\tilde m_{\de,\vep}-Du \cdot \tw_{\de,\vep}]\,  dxdt \\
  = \liminf_{\vep\to 0}  & \int_0^T\into \int_0^T [H(t,x,m,Du)\tmd(s)- Du \cdot \tw_\de(s)]\xi_\vep(s-t)\, dsdxdt
\\
 & \quad  \geq \parint [H(t,x,m,Du)\tmd - Du \cdot \tw_\de ] \, dxdt\,,
\end{align*}
and in the same way 
\begin{equation}\label{liminfde}
\begin{split}
\liminf_{\de\to 0} &    \parint [H(t,x,m,Du)\tmd - Du \cdot \tw_\de ] \, dxdt 
\\
& \qquad \geq 
\parint [H(t,x,m,Du) \tm- Du\cdot H_p(t,x,\tm,D\tu)\tm]\, dxdt\,.
\end{split}
\end{equation}
Now we consider the remaining terms in \rife{prevep}, in particular the terms at $t=0$ and $t=T$. For $t=T$, we have
\begin{align*}
 & \into  \int_{\R} G(x,\lambda)\, d\nu_x(\lambda)\tilde m_{\de,\vep}(T)\, dx 
   = \into  \int_{\R} G(x,\lambda)\, d\nu_x(\lambda)\tmd(T) \,dx 
\\
&\qquad  +
\into \int_0^T \int_{\R} G(x,\lambda)\, d\nu_x(\lambda) [\tmd(s)-\tmd(T)] \xi_\vep(s-T)\, ds dx\,.
\end{align*}
Since $\tm \in \continue 1$ implies that $\int_0^T\ [\tmd(s)-\tmd(T)] \xi_\vep(s-T)\, ds$ converges  to zero uniformly as $\vep\to 0$,  the last integral will vanish as $\vep \to 0$.

For $t=0$, we recall that $\into u(t)\, (\tilde m_0\star \rho_\de)dx$  has a trace at $t=0$ from Lemma \ref{sec:case-non-singular}, and this trace is also continuous as $\de\to 0$ since $\tilde m_0$ is continuous. Therefore, as $\vep \to 0$ we obtain from \rife{prevep}
 \begin{equation}\label{postvep}
\begin{split} & 
\parint [H(t,x,m,Du)\tmd - Du \cdot \tw_\de ] \, dxdt + \into  (\tilde m_0\star \rho_\de) u(0) \,dx
\\
& \quad \leq   \int_0^T\into F(t,x,m)\tmd\, dxdt + \into  \int_{\R} G(x,\lambda)\, d\nu_x(\lambda)\tmd(T) \,dx 
\end{split}
\end{equation}
We apply assumptions \rife{F1} and \rife{G2} to deal with the last two terms. Indeed, if $c_4$ is the constant in \rife{F0}, we have 
$  c_4 \tmd \le  F(t,x,m) \tmd \le   (F(m) +  |c_4|) \tmd$ implies that
\begin{displaymath}
  \begin{split}
   c_4 \tmd \le  F(t,x,m) \tmd & \le    (F(t,x,m) +  |c_4|)  m  \mathds{1}_{m\ge \tilde m_\delta }  + \left( \frac 1 \lambda f(m) +\kappa  +  |c_4| \right) \tmd  
\mathds{1}_{m< \tilde m_\delta }\\
 & \le    (F(t,x,m) +  |c_4|)  m    + \left( \frac 1 \lambda f(\tmd) +\kappa+  |c_4|\right) \tmd  
\\
 & \le    F(t,x,m)m     +        \frac 1 \lambda f(\tmd) \tmd  +  C (m+\tmd) .
  \end{split}
\end{displaymath}
Hence
\begin{align*}
   c_4 \tmd \le F(t,x,m)\tmd & \leq F(t,x,m)m+ \frac1\lambda \, f(\tmd)\tmd + C \, (m+\tmd)
\\ & \leq  F(t,x,m)m+ \frac1\lambda \, (f(m)m )\star \rho_\de + C \, (m+\tmd), 
\end{align*}
and since $F(t,x,m)m\in \parelle1 $ (and the same holds for $f(m)m$ by \rife{F1}), we conclude that
the sequence  $F(t,x,m)\tmd$ is bounded above and below by convergent sequences in $\parelle1$, which allows us to pass to the limit as $\de\to 0$.  We proceed similarly for the term with $G$. Indeed, by \rife{G1},
$$
\int_{\R} G(x,\lambda)\, d\nu_x(\lambda)\tmd(T)   \leq \int_{\R} G(x,\lambda)\lambda \, d\nu_x(\lambda) +   G(x,\tmd(T)) \tmd(T)\quad \hbox{for a.e.   $x\in \Omega$,}
$$
and using \rife{G2} we deduce
\begin{align*}
\int_{\R} G(x,\lambda)\, d\nu_x(\lambda)\tmd(T) & \leq \int_{\R} G(x,\lambda)\lambda \, d\nu_x(\lambda) + \frac1\lambda \, g(\tmd(T))\tmd(T) + C \, \tmd(T)
\\ & \leq  \int_{\R} G(x,\lambda)\lambda \, d\nu_x(\lambda)+ \frac1\lambda \, (g(\tm(T))\tm(T) )\star \rho_\de  + C \, \tmd(T)\,.
\end{align*}
Thanks to \rife{gll} and since $G(x,\tm(T)) \tm(T)\in \elle1$ (and and the same holds for $g(\tm(T))\tm(T)$ by \rife{G2}), we can handle this term too.
Finally, passing to the limit in \rife{postvep} and  using also \rife{liminfde}, we obtain \rife{ineq}.
 
\vskip1em
{\bf (B) Case when $\mathbf \alpha >1$.} \quad First recall from Remark \ref{sec:running-assumptions-2} that  $\alpha\leq \beta$. Moreover, since $1\geq \beta-1$, we also have $\alpha>\beta-1$, i.e. $\beta-\alpha<1$.

We estimate  now differently the first two terms in \rife{prevep}.   First of all, by Young's inequality, we have
 (omitting to write that $\tmd, D\tu$ are evaluated at $s$ and $m,Du$ at $t$)
$$
Du\cdot \tw_\de   \leq \sigma \, \frac{ \max(m+\mu, \tmd)}{\tmd^{\frac{\beta-\alpha}{\beta-1}}}\, |\tw_\de|^{\beta'} + C_\sigma \, \frac{\tmd^{\beta-\alpha}}{\max(m+\mu,\tmd)^{\beta-1} } |Du|^\beta
$$
and using  $\beta\geq \alpha$ and $\alpha>1$ this yields
\begin{equation}\label{stide0}
Du\cdot \tw_\de   \leq \sigma \, \frac{ m+\mu + \tmd}{\tmd^{\frac{\beta-\alpha}{\beta-1} }}\, |\tw_\de|^{\beta'} + C_\sigma \, (m+\mu)^{1-\alpha} |Du|^\beta\,.
\end{equation}
By \rife{grow},
\begin{align*}
|(\tm\, \tb)\star \rho_\de|  & \leq c_2 \tmd + c_2  \left(\tm\, \frac{|D\tu |^{\beta-1}}{(\tm+\mu)^{\alpha}}\right)\star \rho_\de 
\\
& \leq c_2 \tmd + c_2 \left(\left( \frac{|D\tu |^{\beta}}{(\tm+\mu)^{\alpha}}\right)\star \rho_\de\right)^{\frac1{\beta'}} \left(\frac{\tm^\beta}{(\tm+\mu)^{\alpha}}\star \rho_\de\right)^{\frac1\beta}
\\ & 
\leq c_2 \tmd + c_2 \left(\left( \frac{|D\tu |^{\beta}}{(\tm+\mu)^{\alpha}}\right)\star \rho_\de\right)^{\frac1{\beta'}} \tmd^{\frac{\beta-\alpha}\beta}
\end{align*}
where we used $\beta-\alpha\leq 1$ in the last inequality. We deduce that
\be\label{new1}
\frac{|\tw_\de|^{\beta'}}{\tmd^{\frac{\beta-\alpha}{\beta-1} }}  \leq C \tmd^{\frac\alpha{\beta-1}}+ C \left( \frac{|D\tu |^{\beta}}{(\tm+\mu)^{\alpha}}\right)\star \rho_\de\,.
\ee
Combining this information with \rife{stide0}, we deduce that
\begin{equation}\label{stide1}
\begin{split}Du\cdot \tw_\de   & \leq \sigma \, C [m+\mu + \tmd] \left( \frac{|D\tu |^{\beta}}{(\tm+\mu)^{\alpha}}\right)\star \rho_\de \\
& \quad  + C \, [1+ m^{1+\frac\alpha{\beta-1} }+ \tmd^{1+\frac\alpha{\beta-1} }]  + C_\sigma \, (m+\mu)^{1-\alpha} |Du|^\beta\,.
\end{split}
\end{equation}
We first use this inequality to get
\begin{align*}Du\cdot \tw_{\de,\vep}   & \leq \sigma \, C \int_0^T [m+\mu + \tmd] \left( \frac{|D\tu |^{\beta}}{(\tm+\mu)^{\alpha}}\right)\star \rho_\de \, \xi_\vep(s-t)ds\\
& \quad  + C    [1+ m^{1+\frac\alpha{\beta-1} }+ \int_0^T \tmd^{1+\frac\alpha{\beta-1} }\xi_\vep(s-t)ds]  + C_\sigma \, (m+\mu)^{1-\alpha} |Du|^\beta\,.
\end{align*}
Since $\into m(t,x) dx = \into m_0(x) dx$ and $\frac{|D\tu  |^{\beta}}{(\tm+\mu)^{\alpha}}\in \parelle1$,  
$$
\int_0^T m(t) \left( \left(\frac{|D\tu(s) |^{\beta}}{(\tm(s)+\mu)^{\alpha}}\right)\star \rho_\de\right)\,\xi_\vep(s-t)ds\quad \mathop{\to}^{\vep \to 0} \quad m(t) \left(\frac{|D\tu (t)|^{\beta}}{(\tm(t)+\mu)^{\alpha}}\right)\star \rho_\de\qquad \hbox{in $\parelle1$.}
$$
Using also  that $\tmd(s)$ is continuous, we deduce that $Du\cdot \tw_{\de,\vep} $ is dominated by a $L^1$-convergent sequence, so we are allowed to take  $\vep\to 0$ in \rife{prevep}, obtaining 
\begin{equation}\label{postvep2}
\begin{split} & 
\parint [H(t,x,m,Du)\tmd - Du \cdot \tw_\de ] \, dxdt + \into  (\tilde m_0\star \rho_\de) u(0) \,dx
\\
& \quad \leq   \int_0^T\into F(t,x,m)\tmd\, dxdt + \into  \int_{\R}G(x,\lambda)d\nu_x(\lambda)\tmd(T) \,dx \,.
\end{split}
\end{equation}
Thanks to \rife{stide1}, and since the right-hand side is bounded as in the previous case, the above inequality also implies
\begin{align*} & 
\parint H(t,x,m,Du)\tmd  \, dxdt \leq   C+ \sigma \, C \parint  [m+\mu + \tmd] \left( \frac{|D\tu |^{\beta}}{(\tm+\mu)^{\alpha}}\right)\star \rho_\de \, dxdt \\
& \quad  + C \parint  [1+ m^{1+\frac\alpha{\beta-1} }+ \tmd^{1+\frac\alpha{\beta-1} }] \, dxdt  + C_\sigma\parint  (m+\mu)^{1-\alpha} |Du|^\beta \, dxdt \,.
\end{align*}
On account of  the bounds on $m, \tm $ in $\parelle{1+\frac\alpha{\beta-1}}$ and using the properties (ii) in Definition \ref{weaksol}, the last line of the above inequality is uniformly bounded. Therefore, using \rife{coerc}  we deduce that  
\begin{align*}   
\parint \tmd\, \frac{|Du|^\beta}{(m+\mu)^\alpha}   \, dxdt & \leq C+ \sigma \,C  \parint  m_\de \frac{|D\tu |^{\beta}}{(\tm+\mu)^{\alpha}}\, dxdt 
\\ & \quad + \sigma \,C\parint  \tmd \left( \frac{|D\tu |^{\beta}}{(\tm+\mu)^{\alpha}}\right)\star \rho_\de \, dxdt\,,
\end{align*}
where $m_\de= m\star \rho_\de$. We obtain a similar inequality reversing the roles of $(u,m)$ and $(\tu,\tm)$, and by addition we get 
\begin{align*}   &
\parint  m_\de \frac{|D\tu |^{\beta}}{(\tm+\mu)^{\alpha}}\, dxdt +\parint \tmd\, \frac{|Du|^\beta}{(m+\mu)^\alpha}   \, dxdt   \\
& \qquad \leq C+ \sigma \,C  \left\{\parint  m_\de \frac{|D\tu |^{\beta}}{(\tm+\mu)^{\alpha}}\, dxdt + \parint \tmd\, \frac{|Du|^\beta}{(m+\mu)^\alpha}   \, dxdt \right\}
\\ & \quad + \sigma \,C\left\{ \parint  \tmd \left( \frac{|D\tu |^{\beta}}{(\tm+\mu)^{\alpha}}\right)\star \rho_\de \, dxdt + \parint  m_\de \left( \frac{|Du |^{\beta}}{( m+\mu)^{\alpha}}\right)\star \rho_\de \, dxdt\right\}\,,
\end{align*}
hence choosing $\sigma$ sufficiently small  we conclude that
\begin{equation}\label{crossde}
\begin{split}   &
\parint  m_\de \frac{|D\tu |^{\beta}}{(\tm+\mu)^{\alpha}}\, dxdt +\parint \tmd\, \frac{|Du|^\beta}{(m+\mu)^\alpha}   \, dxdt   \\
& \quad \leq C + C \left\{ \parint  \tmd \left( \frac{|D\tu |^{\beta}}{(\tm+\mu)^{\alpha}}\right)\star \rho_\de \, dxdt + \parint  m_\de \left( \frac{|Du |^{\beta}}{( m+\mu)^{\alpha}}\right)\star \rho_\de \, dxdt\right\}\,.
\end{split}
\end{equation}
We wish now to estimate last terms in the above inequality. To this purpose,   we  use \rife{postvep2} with $\tilde m$ replaced by $m_\de$ (and so $\tilde w$ by $ (mH_p(t,x,Du))\star \rho_\de$). Since the terms with $F$ and $G$ can be estimated as before, we get
$$
\parint  m_\de\, [H(t,x,m,Du) \star \rho_\de]  \, dxdt \leq \parint  [Du\star \rho_\de] \cdot [(mH_p(t,x,Du))\star \rho_\de] \, dxdt + C\,.
$$
Therefore, using \rife{coerc} we obtain
\begin{align*} & 
c_0 \parint  m_\de [\left( \frac{|Du |^{\beta}}{( m+\mu)^{\alpha}}\right)\star \rho_\de] \, dxdt \leq  \parint  [Du\star \rho_\de] \cdot [(mH_p(t,x,Du))\star \rho_\de] \, dxdt + C \\
&\qquad  \leq c_2\parint  |Du\star \rho_\de |\,  \left(m\frac{|Du|^{\beta-1}}{(m+\mu)^\alpha}\right)\star \rho_\de  \, dxdt  + C \parint  \left |Du\star \rho_\de\right| \, m_\de \, dxdt 
\\ & 
\qquad  \leq c_2 \parint  |Du\star \rho_\de |\,  \left(\left(\frac{|Du|^{\beta}}{(m+\mu)^\alpha}\right)\star \rho_\de\right)^{\frac1{\beta'}} m_\de^{\frac{\beta-\alpha}\beta} \, dxdt  + C \parint  \left |Du\star \rho_\de\right| \, m_\de \, dxdt
\\ & \qquad \leq \frac{c_0}2  \parint  (m+\mu)\star \rho_\de \left( \frac{|Du |^{\beta}}{( m+\mu)^{\alpha}}\right)\star \rho_\de \, dxdt\\ & \qquad \quad + C \parint  |Du\star \rho_\de|^\beta\, ((m+\mu)\star \rho_\de)^{1-\alpha} \, dxdt\ + C\parint 1+m_\de^{1+\frac\alpha{\beta-1}}dxdt \,.
\end{align*}
Hence we conclude that
$$
\parint  m_\de \left( \frac{|Du |^{\beta}}{( m+\mu)^{\alpha}}\right)\star \rho_\de \, dxdt  \leq C \parint  |Du\star \rho_\de|^\beta\, ((m+\mu)\star \rho_\de)^{1-\alpha} \, dxdt\ + C \,.
$$
Now we observe that the function $(m,\xi) \mapsto (m+\mu)^{1-\alpha}|\xi|^\beta$ is convex as  a function of two variables, therefore
$$
 |Du\star \rho_\de|^\beta\, ((m+\mu)\star \rho_\de)^{1-\alpha} \leq  \left((m+\mu)^{1-\alpha} |Du|^{\beta}\right)\star \rho_\de
$$
which is bounded since $(m+\mu)^{1-\alpha} |Du|^\beta\in \parelle1$. We deduce the uniform bound (with respect to $\de$)
$$
\parint  m_\de \left( \frac{|Du |^{\beta}}{( m+\mu)^{\alpha}}\right)\star \rho_\de \, dxdt  \leq C\,,
$$
and similarly for $\tm$. 
Going back to \rife{crossde}, this also implies the bound
$$
\parint  m_\de \frac{|D\tu |^{\beta}}{(\tm+\mu)^{\alpha}}\, dxdt +\parint \tmd\, \frac{|Du|^\beta}{(m+\mu)^\alpha}   \, dxdt \leq C\,.
$$
Now, reasoning as in \rife{stide0}--\rife{stide1}, 
\begin{align*}
\int\!\!\!\int_E |Du\cdot \tw_\de  |  & \leq  \left(\int\!\!\!\int_E \frac{  (m+\mu+ \tmd)}{\tmd^{\frac{\beta-\alpha}{\beta-1}}}\, |\tw_\de|^{\beta'} \right)^{\frac1{\beta'}}\left( \int\!\!\!\int_E (m+\mu)^{1-\alpha} |Du|^\beta\right)^{\frac1\beta}
\\  & \leq  \left(C+ \int\!\!\!\int_E  [m+\mu + \tmd] \left( \frac{|D\tu |^{\beta}}{(\tm+\mu)^{\alpha}}\right)\star \rho_\de\right)^{\frac1{\beta'}}\left( \int\!\!\!\int_E (m+\mu)^{1-\alpha} |Du|^\beta\right)^{\frac1\beta}\,.
\end{align*}
The bounds previously established yield
$$
\int\!\!\!\int_E |Du\cdot \tw_\de  |   \leq  C \left( \int\!\!\!\int_E(m+\mu)^{1-\alpha} |Du|^\beta\right)^{\frac1\beta}
$$
and since the set  $E$ is arbitrary and $(m+\mu)^{1-\alpha} |Du|^\beta\in \parelle1$, we deduce  the equi-integrability of $Du\cdot \tw_\de $. Finally, this allows us to pass to the limit in \rife{postvep2} and to obtain \rife{ineq}.
\qed

A similar Lemma holds for the case of singular congestion, suitably adapted to the formulation of this case.

\begin{lemma}\label{compare-s}
Assume that $H$ satisfies assumptions \rife{coerc}-\rife{grow} with $\mu=0$, and that \rife{F1}--\rife{G2} hold true.  
Let  $u\in \limitate1$  satisfy 
\begin{align*}
\int_0^T \into u\, \vfi_t\, dxdt & - \nu \int_0^T \into u\,\Delta\vfi\, dxdt 
+ \int_0^T\into H(t,x,m,Du)\mathds{1}_{\{m>0\}}\,\vfi\, dxdt
\\
& \leq  \int_0^T \into F(t,x,m)\vfi\, dxdt + \into \int_{\R} G(x,\lambda)d\nu_x(\lambda)\,\vfi(T)\, dx
\end{align*}
for some family of probability measures $\{\nu_x\}$ (weakly$-*$ measurable w.r.t. $x$) such that  \rife{gll} is satisfied, and  let $m\in \continue1$ be a solution of  \rife{eq:37}. Assume that $(u,m)$ satisfy the conditions (ii) 
in Definition \ref{defsing}.
Then 
\begin{equation}\label{ineq-s}
\begin{split}
\langle \tilde m_0\,, u(0) \rangle & \leq  \into \int_{\R} G(x,\lambda)d\nu_x(\lambda)\, \tilde m(T)\,dx + \parint F(t,x,m)\tilde m\, dxdt
\\ & + \parint \left[ \tilde m\, H_p(t,x,\tilde m, D\tilde u)\cdot Du- \tilde m \,H(t,x,m,Du)\right]\, \mathds{1}_{\{m>0\,,\, \tm>0\}}dxdt
\end{split}
\end{equation}
for any couple $(\tilde u, \tilde m)$ satisfying the same conditions as $(u,m)$.
\end{lemma}

\begin{proof}
  The proof follows the same lines as that of Lemma~\ref{compare}.\\
{\bf (A) Case when $\mathbf \alpha\leq 1$.}\quad
Proceeding as in the proof of Lemma ~\ref{compare}, we obtain the counterparts of (\ref{ueq})
\begin{equation}\label{eq:23}
\begin{split} & 
 \parint u \,[\partial_t \tilde m_{\de,\vep}-\nu \Delta \tilde m_{\de,\vep}]\, dxdt  + \parint  H(t,x,m,Du)\tilde m_{\de,\vep}\, \mathds{1}_{\{m>0\}} \, dxdt
\\
& \quad \leq \parint  F(t,x,m)\tilde m_{\de,\vep}\,dxdt  + \into \int_{\R} G(x,\lambda)d\nu_x(\lambda)\tilde m_{\de,\vep}(T)\, dx 
\end{split}
\end{equation}
and of (\ref{prevep}): 
\begin{equation}\label{eq:24}
\begin{split} & 
-\parint   Du \cdot \tw_{\de,\vep} \, dtdx  + \parint  H(t,x,m,Du)\tilde m_{\de,\vep}\,  \mathds{1}_{\{m>0\}}\, dxdt + \into  (m_0\star \rho_\de) \int_0^T u(t)\xi_\vep(-t) \, dt\,\,dx
\\ & \quad  
\leq  \parint  F(t,x,m)\tilde m_{\de,\vep}\,dxdt  + \into  \int_{\R} G(x,\lambda)d\nu_x(\lambda)\tilde m_{\de,\vep}(T)\, dx 
\end{split}
\end{equation}
where we denote  $\tw_\de= [( \tm \, H_p(t,x,\tm, D\tu) \,\mathds{1}_{\{\tm>0\}} )\star \rho_\de]$ and $\tw_{\de,\vep}= \int_0^T \tw_\de(s)\, \xi_\vep(s-t)\,ds$.  
Note that  $\tw_\de= [( \tm \, H_p(t,x,\tm, D\tu)  )\star \rho_\de]$ if $\alpha<1$. Set $\tb=  H_p(t,x,\tm, D\tu) \,\mathds{1}_{\{\tm>0\}} $.
We deduce from Assumption (\ref{grow}) that 
\begin{align*}
|(\tm\, \tb)\star \rho_\de| & \leq c_2 \tmd + c_2 \left(\left(\tm\, \frac{|D\tu |^{\beta}}{\tm^{\alpha}} \mathds{1}_{\{\tm>0\}}\,        \right)\star \rho_\de\right)^{\frac1{\beta'}} \left(\frac\tm{\tm^{\alpha}}\star \rho_\de\right)^{\frac1\beta}
\\ &
\leq c_2 \tmd + c_2 \left(\left(\tm ^{1-\alpha} \, |D\tu |^{\beta}    \mathds{1}_{\{\tm>0\}}  \right)
\star \rho_\de\right)^{\frac1{\beta'}}  (\tm\star \rho_\de)^{\frac{1-\alpha}\beta}
\end{align*}
where we used $\alpha\leq 1$ in the latter inequality. Therefore, using also \rife{coerc} we estimate 
\begin{equation}
  \label{eq:30}
  \begin{split}
& H(t,x,m,Du)\mathds{1}_{\{m>0\}}\tmd(s)- Du \cdot \tw_\de(s) \geq 
c_0 \tmd(s) \frac{|Du|^{\beta}}{m^{\alpha}}   \mathds{1}_{\{m>0\}}  - c_1  \tmd(s) (1+m^{\frac\alpha{\beta-1}}) 
\\
& - c_2 \tmd(s) |Du| -
c_2 |Du|\, \left(\left(\tm(s) ^{1-\alpha} \, |D\tu(s) |^{\beta}  \mathds{1}_{\{\tm(s)>0\}}\right) \star \rho_\de\right)^{\frac1{\beta'}}\tmd(s)^{\frac{1-\alpha}\beta}.    
  \end{split}
\end{equation}
We now use the fact that $Du=0$ a.e. in $\{m=0\}$. From Young inequalities, we deduce 
\begin{equation}
  \label{eq:26}
  c_2 \tmd(s) |Du|= 
  c_2 \tmd(s) |Du|  \mathds{1}_{\{m>0\}}  \le \frac {c_0}2 \tmd(s)  \frac {|Du|^\beta}{m^\alpha}  \mathds{1}_{\{m>0\}} + C \tmd(s)  m^{\frac \alpha {\beta -1}},
\end{equation}
and
\begin{equation}
  \label{eq:27}
\begin{split}  
  c_2 |Du|\, \left(\left(\tm(s) ^{1-\alpha} \, |D\tu(s) |^{\beta} \mathds{1}_{\{\tm(s)>0\}} \right) \star \rho_\de\right)^{\frac1{\beta'}}\tmd(s)^{\frac{1-\alpha}\beta}
\\    \le  \sigma |Du|^\beta  \mathds{1}_{\{m>0\}}  \tmd(s)^{1-\alpha} 
    + C_\sigma  \left(\tm(s) ^{1-\alpha} \, |D\tu(s) |^{\beta} \mathds{1}_{\{\tm(s)>0\}} \right)  \star \rho_\de ,    
  \end{split}
\end{equation}
and if $0<\sigma\le c_0/2$,
\begin{equation}
\label{eq:28}
\sigma  \mathds{1}_{\{m>0\}}  \tmd(s)^{1-\alpha} \le \frac {c_0}2  
\frac {\tmd(s)}{m^\alpha}  \mathds{1}_{\{m>0\}} +\frac {c_0}2   m^{1-\alpha}  \mathds{1}_{\{m>0\}},  
\end{equation}
because $0<\alpha\le 1$.  Finally,
\begin{equation}
\label{eq:29}
\tmd(s)  m^{\frac \alpha {\beta -1}} \le C\left( 1 + m^{1+\frac \alpha {\beta -1}}   + \tmd(s)^{1+\frac \alpha {\beta -1}}   \right).  
\end{equation}
Combining~(\ref{eq:30}) with (\ref{eq:26})-(\ref{eq:29}), we get 
\begin{align*}
& H(t,x,m,Du)\tmd(s)- Du \cdot \tw_\de(s) \geq  - C\,  \left(1+ \tmd(s)^{1+\frac{\alpha}{\beta-1}}+m^{1+\frac\alpha{\beta-1}}\right) 
\\
&  - \frac {c_0} 2 \, m^{1-\alpha} \,  |Du|^\beta\, \mathds{1}_{\{m>0\}}   - c \,\left(   \tm(s)  ^{ 1 -\alpha}\, |D\tu(s) |^{\beta} \mathds{1}_{\{\tm(s)>0\}} \right)\star \rho_\de \,.
\end{align*}
The conclusion of the proof is exactly as for Lemma \ref{compare}.
\vskip1em
{\bf (B) Case when $\mathbf \alpha >1$.} \quad
We start from (\ref{eq:24}), and proceed differently for the first  two terms. From \rife{grow} we have  $|\tilde b| \leq  c_2 (1+ |D\tilde u|^{\beta-1}\tilde m^{-\alpha}\mathds{1}_{\{\tm >0\}})$, so we estimate $\tilde w_\de$ as in \rife{new1} of Lemma \ref{compare} and we get
\begin{displaymath}
\frac {|\tw_\de|}{ \tmd ^{\frac {\beta-\alpha} {\beta-1}}}\le C \tmd ^{\frac {\alpha} {\beta-1}}+C 
\left( \frac{|D\tu |^{\beta}}{\tm^{\alpha}} \mathds{1}_{\{\tm>0\}}\right)\star \rho_\de.  
\end{displaymath}
Since $Du=0$ a.e. if $m=0$, we can use \rife{stide0} with $\mu=0$ and, similarly as in Lemma \ref{compare}, we get at 
\begin{displaymath}
  \begin{split}
    Du\cdot \tw_\de =\;&Du\cdot  \frac { \tw_\de}  {\tmd ^{\frac {\beta-\alpha} {\beta}}}   \tmd ^{\frac {\beta-\alpha} {\beta}}  \mathds{1}_{\{m>0\}}\\
    \le \;  & \sigma C (\tmd+ m  \mathds{1}_{\{m>0\}}) \left( \frac{|D\tu |^{\beta}}{\tm^{\alpha}} \mathds{1}_{\{\tm>0\}}\right)\star \rho_\de
    +C\left(1+  m^{1+\frac \alpha {\beta-1} } +\tmd^{1+\frac \alpha {\beta-1} } \right)\\ &+C_\sigma    m ^{1-\alpha}    \mathds{1}_{\{m>0\} }  |Du|^\beta .
  \end{split}
\end{displaymath}
Thanks to this estimate, we can proceed exactly as in the proof of Lemma~\ref{compare} (when $\alpha>1$), and obtain 
\begin{equation}\label{eq:25}
\begin{split} & 
\parint [H(t,x,m,Du)  \mathds{1}_{\{m>0\} }  \tmd - Du \cdot \tw_\de ] \, dxdt + \into  (\tilde m_0\star \rho_\de) u(0) \,dx
\\
& \quad \leq   \int_0^T\into F(t,x,m)\tmd\, dxdt + \into  \int_{\R} G(x,\lambda)d\nu_x(\lambda)\tmd(T) \,dx \,.
\end{split}
\end{equation}
and, after some steps,
\begin{equation}\label{eq:35}
\begin{split}   &
\parint  m_\de \frac{|D\tu |^{\beta}}{\tm^{\alpha}} \mathds{1}_{\{\tm>0\} }\, dxdt +\parint \tmd\, \frac{|Du|^\beta}{m^\alpha}  \mathds{1}_{\{m>0\} } \, dxdt   \\
& \quad \leq C + C \left\{ \parint  \tmd \left( \frac{|D\tu |^{\beta}}{\tm^{\alpha}} \mathds{1}_{\{\tm>0\} }\right)\star \rho_\de \, dxdt + \parint  m_\de \left( \frac{|Du |^{\beta}}{ m^{\alpha}}  \mathds{1}_{\{m>0\} }\right)\star \rho_\de \, dxdt\right\}\,.
\end{split}
\end{equation}
To estimate the last two terms in (\ref{eq:35}), we use (\ref{eq:25}) where we replace $\tm$ with $m_\delta$ and proceed as above, 
\begin{align*} & 
c_0 \parint  m_\de \left( \frac{|Du |^{\beta}   \mathds{1}_{\{m>0\} } }{ m^{\alpha}}\right)\star \rho_\de \, dxdt
 \leq  \parint  Du\star \rho_\de \cdot (mH_p(t,x,Du)  \mathds{1}_{\{m>0\} })\star \rho_\de \, dxdt + C \\
&\qquad  \leq c_2\parint  |Du\star \rho_\de |\,  \left(m^{1-\alpha} |Du|^{\beta-1}     \mathds{1}_{\{m>0\} } \right)\star \rho_\de  \, dxdt  + C 
\\ & 
\qquad  \leq c_2 \parint  |Du\star \rho_\de |\,  \left(\left(\frac {|Du|^{\beta}}{m^\alpha}  \mathds{1}_{\{m>0\} }\right)\star \rho_\de\right)^{\frac1{\beta'}}
  (m\star \rho_\de)^{\frac{\beta-\alpha}\beta}  \, dxdt  + C
  \end{align*}
  Let us observe that $m\star \rho_\de =0 \Rightarrow Du\star \rho_\de =0$ since  $Du=0$ at almost every point where $m=0$. This means that last integral in the above inequality is restricted to the set $\{m\star \rho_\de>0\}$ (even in the limiting case $\beta=\alpha$). By Young's inequality, we deduce that
\begin{equation}
  \label{eq:31}
 \parint  m_\de \left( \frac{|Du |^{\beta}}{m^{\alpha}} \mathds{1}_{\{m>0\} }   \right)\star \rho_\de \, dxdt\le   C \parint  |  Du \star \rho_\de|^\beta\, (m\star \rho_\de)^{1-\alpha}   \mathds{1}_{\{m\star \rho_\de>0\} }   \, dxdt\ + C .
\end{equation}
Let us define the convex and lower semi-continuous function $\Psi$ on $\R^N\times \R$ by
\begin{displaymath}
  \Psi(p,m)= \left\{
    \begin{array}[c]{ll}
      m^{1-\alpha} |p|^\beta  \quad & \hbox{if}\quad m>0,\\
      0\quad & \hbox{if}\quad m=0 \hbox{ and } p=0,\\
      +\infty \quad &\hbox{otherwise}.
    \end{array}
\right.
\end{displaymath}
This implies that 
\begin{displaymath}
  \begin{split}
 (m\star \rho_\de)^{1-\alpha}   \, |  Du\star \rho_\de|^\beta\,  \mathds{1}_{\{m\star \rho_\de>0\} }
 &=  \Psi (Du\star \rho_\de, m\star \rho_\de)    \\
& \le \Psi (Du, m) \star \rho_\de \\
&=  \left( m^{1-\alpha}   \,  |Du|^\beta\,  \mathds{1}_{\{m>0\} }\right) \star \rho_\de
  \end{split}
\end{displaymath}
and the latter function is bounded in $L^1(Q_T)$ by assumption. We deduce from (\ref{eq:31}) the uniform bound (with respect to $\delta$)
\[  \parint  m_\de \left( \frac{|Du |^{\beta}}{m^{\alpha}} \mathds{1}_{\{m>0\} }   \right)\star \rho_\de \, dxdt \le C,\]
and the same holds replacing $m$ with $\tm$ and $u$ with $\tu$. Going back to (\ref{eq:35}), we obtain 
\begin{displaymath}
  \parint  m_\de \frac{|D\tu |^{\beta}}{\tm^{\alpha}} \mathds{1}_{\{\tm>0\} }\, dxdt +\parint \tmd\, \frac{|Du|^\beta}{m^\alpha}  \mathds{1}_{\{m>0\} } \, dxdt
   \leq C, \end{displaymath}
and we conclude as in the proof of Lemma~\ref{compare}.
\end{proof}

\section{Existence and uniqueness for non singular congestion}

Our first goal is to prove the existence of a weak solution and, to this purpose, we now show the strong convergence of $G^\epsilon(x,m^\epsilon(T))$ in $\elle1$. 

\label{sec:exist-uniq-non}
\begin{lemma}\label{sec:exist-uniq-non-2}
Consider a subsequence $(u^\epsilon, m^\epsilon)$ converging to $(u,m)$ as in  Proposition \ref{sec:non-sing-cong-5}.
 We have that $G^\epsilon(x,m^\epsilon(T))$ converges to $G(x,m(T))$ in $\elle1$ and $(u,m)$ satisfies the energy identity:
\begin{equation}
  \label{eq:22}
    \begin{split}
  \parint m \left( H_p(t,x, m,Du)\cdot Du -H(t,x, m,Du)\right) \,dxdt& \\ +\into m(T) G(x,m(T)) dx+\parint mF(t,x,m)\, dxdt&=    \langle u(0) ,m_0 \rangle.      
    \end{split}
\end{equation}
\end{lemma}

\begin{proof}
We start from the energy identity for the solution of (\ref{eq:3u})-(\ref{eq:3bc}):
\be\label{eneps}
\begin{split}& 
 \parint m^\epsilon \left( H_p(t,x, T_{1/\epsilon}m^\epsilon,Du^\epsilon)\cdot Du^\epsilon -H(t,x, T_{1/\epsilon}m^\epsilon,Du^\epsilon)\right) \,dxdt
\\
&  +\into m^\epsilon(T) G^\epsilon(x,m^\epsilon(T)) dx+\parint m^\epsilon F^\epsilon(t,x,m^\epsilon)\, dxdt
\,\, = \into u^\epsilon(0)m_0^\epsilon{}\, dx .
\end{split}
\ee
Thanks to (\ref{eq:2}) and (\ref{F0}), Fatou's lemma implies that 
\begin{displaymath}
  \begin{split}
 & \liminf_{\epsilon\to 0} \parint m^\epsilon \left( H_p(t,x, T_{1/\epsilon}m^\epsilon,Du^\epsilon)\cdot Du^\epsilon -H(t,x, T_{1/\epsilon}m^\epsilon,Du^\epsilon)\right) \,dxdt \\  \ge&  \parint m \left( H_p(t,x, m,Du)\cdot Du -H(t,x, m,Du)\right) \,dxdt,
  \end{split}
\end{displaymath}
and
$$
\liminf_{\epsilon\to 0} \parint m^\epsilon F^\epsilon(t,x,m^\epsilon)\, dxdt \ge  \parint mF(t,x,m)\, dxdt.
$$
{From} Lemma \ref{sec:non-sing-cong-8}  we can assume that  $u^\epsilon|_{t=0}$ converges weakly * to a bounded measure $\chi$  on $\Omega$, and $\chi\leq u(0)$. Since $m_0\in C(\Omega)$, we deduce that
 $$ 
 \lim_{\epsilon\to 0} \into u^\epsilon(0)m_0^\epsilon\, dx=
\langle \chi ,m_0 \rangle \leq \langle u(0) ,m_0 \rangle\,.
$$
 Combining  the informations above, we get from \rife{eneps}
\begin{align*} 
 \limsup\limits_{\epsilon \to 0} & \into m^\epsilon(T) G^\epsilon(x,m^\epsilon(T)) dx \leq  \langle u(0) ,m_0 \rangle -\parint mF(t,x,m)\, dxdt
 \\
 & - \parint m \left( H_p(t,x, m,Du)\cdot Du -H(t,x, m,Du)\right) \,dxdt
\end{align*}
We now use~(\ref{ineq}) in Lemma~\ref{compare} with $\tilde m=m$ and $\tilde u=m$, and we get
\be\label{geps}
\begin{split} 
 \limsup\limits_{\epsilon \to 0} & \into m^\epsilon(T) G^\epsilon(x,m^\epsilon(T)) dx \leq  \langle u(0) ,m_0 \rangle -\parint mF(t,x,m)\, dxdt
 \\
 & - \parint m \left( H_p(t,x, m,Du)\cdot Du -H(t,x, m,Du)\right) \,dxdt
 \\
 & \qquad \leq  \into \int_{\R} m(T) G(x,\lambda)d\nu_x(\lambda) dx\,.
\end{split}
\ee
We now use the monotonicity of $G$ in order to get the strong convergence of $G^\epsilon(x,m^\epsilon(T))$. Indeed, if we set $\hat m^\epsilon:= \rho^\epsilon \star m^\epsilon(T)$ and $T_k(r)= \min(r,k)$, 
\begin{align*}
 \into & \left[ T_k(G(x,\hat m^\epsilon))- T_k(G(x,m(T)))\right][\hat m^\epsilon-m(T)] dx
 \\
&  \leq   \into  G(x,\hat m^\epsilon) \hat m^\epsilon \, dx -  \into  T_k(G(x,\hat m^\epsilon)) m(T)\, dx  - \into T_k(G(x,m(T))) \, [\hat m^\epsilon-m(T)] dx\,.
\end{align*}   
Since $\hat m^\epsilon$ weakly converges to $m(T)$, the last term vanishes as $\epsilon \to 0$.
The first one is estimated by \rife{geps}, while for the second we can use \rife{proyou} with $f(x,\cdot)= T_k(G(x,\cdot))$.  Therefore, we get
\begin{align*}
 \limsup_{\epsilon \to 0}\into & \left[ T_k(G(x,\hat m^\epsilon))- T_k(G(x,m(T)))\right][\hat m^\epsilon-m(T)] dx
 \\
&  \leq   \into \int_{\R} G(x,\lambda) m(T) d\nu_x(\lambda) \, dx -  \into  \int_{\R} T_k(G(x,\lambda)) m(T) d\nu_x(\lambda) \, dx\,.
\end{align*}   
Letting $k\to \infty$ we  conclude that 
\be\label{kG}
\limsup_{k \to \infty }\,\, \limsup_{\epsilon \to 0}\into  \left[ T_k(G(x,\hat m^\epsilon))- T_k(G(x,m(T)))\right][\hat m^\epsilon-m(T)] dx =0\,.
\ee
We claim now that, as a  consequence of \rife{kG}, $G(x,\hat m^\epsilon)$ converges to $G(x,m(T))$ almost everywhere in $\Omega$, at least for a subsequence. Indeed, we observe that 
$$
\begin{array}{c}
\into \frac{ [G(x,\hat \me )-G (x,  m(T))] \, [ \hat m^\vep  - m(T)]}{1+\hat m^\vep + m(T)} dx\leq 
\into  \left[ T_k(G(x,\hat m^\epsilon))- T_k(G(x,m(T)))\right][\hat m^\epsilon-m(T)] dx
\\
\m
  + \into |G (x,  \hat m^\epsilon )| \, \mathds{1}_{\{G (x,  \hat m^\epsilon )>k\} }\, dx
+ \into |G (x,  m(T))| \, \mathds{1}_{\{G (x,  m(T))>k\} }\, dx
\end{array}
$$
and last two terms tend to zero as $k\to \infty$ uniformly with respect to $\vep$.  So 
we have
$$
\begin{array}{c}
\limsup\limits_{\epsilon \to 0}\into \frac{ [G(x,\hat \me )-G (x,  m(T))] \, [ \hat m^\vep  - m(T)]}{1+\hat m^\vep + m(T)}  dx 
\\
\m \quad \leq 
\limsup\limits_{\epsilon \to 0}\into  \left[ T_k(G(x,\hat m^\epsilon))- T_k(G(x,m(T)))\right][\hat m^\epsilon-m(T)] dx
\\
\m
+ \sup_{\vep} \,\,  \into |G (x,  \hat m^\epsilon  )| \, \mathds{1}_{\{G (x,  \hat m^\epsilon )>k\} }\,  dx + \into |G (x,  m(T))| \, \mathds{1}_{\{G (x,  m(T))>k\} }\, dx
\end{array}
$$
hence, using \rife{kG} and  letting $k\to \infty$, the right-hand side vanishes. We deduce that 
$$
\limsup_{\epsilon \to 0}\into \frac{  [G(x,\hat \me )-G (x,  m(T))] \, [ \hat m^\vep  - m(T)]}{1+\hat m^\vep + m(T)} dx= 0
$$
which means, since $G$ is monotone,  the $L^1$ convergence of the integrand function. We deduce that, up to subsequences, 
$$
\frac{ [G(x,\hat \me )-G (x,  m(T))] \, [ \hat m^\vep  - m(T)] }{1+\hat m^\vep + m(T)}  \to 0 \quad \hbox{a.e.}
$$
and this readily yields the a.e. convergence of $G(x,\hat m^\epsilon)$  to $G(x,m(T))$  in $\Omega$. 

Finally, since $G(x,\hat m^\epsilon)$ is also equi-integrable by \rife{eq:11}, it is therefore convergent in $\elle1$. As a consequence, the $L^1$ convergence of $G^\vep(x, m^\epsilon(T))$ towards $G(x,m(T))$ is established. In addition, we now deduce that 
$$
\int_{\R}G(x,\lambda)d\nu_x(\lambda)= G(x,m(T))\,.
$$
We insert this information in the right-hand side of \rife{geps}, moreover  now we can use Fatou's lemma in the left-hand side and we conclude with the identity
\begin{align*} 
 \into G(x,m(T))m(T) dx & =  \langle u(0) ,m_0 \rangle -\parint mF(t,x,m)\, dxdt
 \\
 & - \parint m \left( H_p(t,x, m,Du)\cdot Du -H(t,x, m,Du)\right) \,dxdt
\end{align*}
which is \rife{eq:22}.
\end{proof}

\begin{remark}
The monotonicity of $G$ is only used in the above lemma in order to obtain the strong $L^1$ convergence of $G^\vep(x, m^\epsilon(T))$.  If $q<2$, this condition would not be required since one already knows  from Proposition \ref{sec:non-sing-cong-5} that $m^\epsilon(T)$ strongly converges in $L^1$ to $m(T)$.
\end{remark}

We can finally conclude with the existence result.

  \begin{Theorem}\label{sec:exist-uniq-non-1}
    Consider a subsequence $(u^\epsilon, m^\epsilon)$ converging to $(u,m)$ as in  Proposition \ref{sec:non-sing-cong-5}.
 Then $(u,m)$ is a weak solution of (\ref{MFGu}).
  \end{Theorem}

  \begin{proof}
With the results proved in \S~\ref{sec:non-sing-cong-6}, there only remains to 
pass to the limit in the Hamilton-Jacobi equation and show that (\ref{MFGu}) holds.\\
By Lemma \ref{sec:exist-uniq-non-2},  $G^\epsilon(x,m^\epsilon(T))\to  G(x,m(T))$ in $L^1(\Omega)$. We also know that $ F^\epsilon(t,x,m^\epsilon) \to F(t,x,m)$ in $L^1(Q_T)$ (see the proof of Lemma~\ref{sec:non-sing-cong-4}). Now we observe that assumptions \rife{eq:1}-\rife{grow} imply that  $H$ satisfies
$$
c_0\,  \frac{|p|^\beta}{(m+\mu)^\alpha}- c_1 \,(1+m^{\frac\alpha{\beta-1}}) \leq H(t,x,m,p) \leq c  \left( |p| + \frac{|p|^\beta}{(m+\mu)^\alpha}\right) \leq c  \left(\frac{|p|^\beta}{(m+\mu)^\alpha}+  1+m^{\frac\alpha{\beta-1}}\right)\,.
$$
Therefore, up to addition of  a $L^1$-convergent sequence, the Hamiltonian 
$$
(t,x,p)\mapsto H(t,x,T_{1/\epsilon}m^\epsilon(t,x), p)
$$
is nonnegative and has  natural growth. 
Since $u^\epsilon$ is bounded below, without loss of generality we can assume that the solution $u^\epsilon$ 
is also nonnegative. It is therefore possible to apply (a  straightforward adaptation of) the result in  \cite[Theorem 3.1]{P}, in order to deduce that   $H(t,x, T_{1/\epsilon}m^\epsilon,Du^\epsilon)  \to H(t,x, m,Du)$ in $L^1 (Q_T)$, 
and  we can pass to the limit in (\ref{eq:3u})-(\ref{eq:3bc}) and obtain (\ref{eq:17}).
 \end{proof}

In order to prove the uniqueness of weak solutions, the main step was given by Lemma \ref{compare}. We need however a counterpart which ensures that {\it any} weak solution satisfies the energy equality \rife{eq:22}. To this purpose, we follow an argument developed  in \cite{Pumi}.

\begin{lemma}\label{enidweak} Let $(u,m)$ be any weak solution of (\ref{MFGu})-(\ref{MFGbc}). Then the energy identity \rife{eq:22} holds true.
\end{lemma}

\begin{proof} 
By Lemma \ref{compare}, we already know that 
\begin{align*}
\into  m_0\,  u(0)\, dx & \leq  \into G(x,m(T))\,  m(T)\,dx + \parint F(t,x,m) m\, dxdt
\\ & + \parint \left[ m\, H_p(t,x,m, D u)\cdot Du- m \,H(t,x,m,Du)\right]dxdt
\end{align*}
where we used the fact that $u\in \continue1$.  In order to prove the reverse inequality, let $u_k:= \min(u,k)$. Since $-u_t -\nu \Delta u \in \parelle1$, by Kato's inequality,
$$
-(u_k)_t -\nu\Delta u_k + H(t,x,m,Du)\mathds{1}_{\{u<k\}}\geq F(t,x,m) \mathds{1}_{\{u<k\}}\,.
$$
Since $m_t-\nu \Delta m -\dive(m\,b)=0$ for some $b$ such that $m|b|^2\in \parelle1$, by \cite[Theorem 3.6]{Parma} we know that $m$ is also a renormalized solution, hence  it satisfies 
$$
S_n(m)_t-\nu \Delta S_n(m) -\dive(S_n'(m)m\,H_p(t,x,m,Du))=\omega_n
$$  
where $S_n(m)$ is a suitable $C^1$ truncation (i.e. $S_n(r)= nS(r/n)$, for some $S$ compactly supported in $[-2,2]$ such that $S=1$ in $[-1,1]$) and where  $\omega_n\to 0$ in $\parelle1$. Notice that both $u_k$ and $S_n(m)$ are bounded functions and belong to $L^2(0,T; H^1)\cap \continue1$. By using $u_k$ in the equation of $S_n(m)$ we get
\begin{align*}
& \into S_n(m_0) u_k(0)\, dx - \into  u_k(T)\, S_n(m(T))\, dx \geq 
\parint F(t,x,m) \mathds{1}_{\{u<k\}}S_n(m)\, dxdt 
\\
& \, + \parint \left[ S_n'(m)m \, H_p(t,x,m,Du)\cdot Du - S_n(m) H(t,x,m,Du)\right] \mathds{1}_{\{u<k\}} \, dxdt - \parint \omega_n\, u_k\, dxdt\,.
\end{align*}
Letting first $n\to \infty$, the last term vanishes since $u_k$ is bounded. The regularity of weak solutions allows us to pass to the limit in the other terms (recall that $u(T)=G(x,m(T))$) and we obtain
 \begin{align*}
& \into m_0\, u_k(0)\, dx - \into  u_k(T)\, m(T)\, dx \geq 
\parint F(t,x,m) \mathds{1}_{\{u<k\}}m\, dxdt 
\\
& \, + \parint \left[ m \, H_p(t,x,m,Du)\cdot Du - m H(t,x,m,Du)\right] \mathds{1}_{\{u<k\}} \, dxdt\,.
\end{align*}  
Finally, letting $k\to \infty$ we deduce the desired inequality and we conclude that \rife{eq:22} holds true.
\end{proof}

  \begin{lemma}
    \label{sec:exist-uniq-non-3}
Under all the assumptions made in Theorem \ref{sec:case-non-singular-5}, there is a unique weak solution of (\ref{MFGu})-(\ref{MFGbc}).
  \end{lemma}
  \begin{proof}
Let $(u,m)$ and $(\tilde u, \tilde m)$ be two weak solutions of (\ref{MFGu})-(\ref{MFGbc}).
By Lemma \ref{enidweak},  they both satisfy the energy identity (\ref{eq:22}), so using also Lemma \ref{compare}  overall we know that 
    \begin{displaymath}
      \begin{split}
        \into  m_0\, u(0)\, dx & \leq  \into G(x,m(T))\, \tilde m(T)\,dx + \parint F(t,x,m)\tilde m\, dxdt
\\ & + \parint \left[ \tilde m\, H_p(t,x,\tilde m, D\tilde u)\cdot Du- \tilde m \,H(t,x,m,Du)\right]dxdt,
\\
        \into \ m_0\, \tilde u(0)\, dx & \leq  \into G(x,\tilde m(T))\,  m(T)\,dx + \parint F(t,x,\tilde m) m\, dxdt
\\ & + \parint \left[  m\, H_p(t,x, m, D u)\cdot D\tilde u-  m \,H(t,x,\tilde m,D\tilde u)\right]dxdt,
\\
  \into  m_0\, u(0)\, dx & =  \into G(x,m(T))\, m(T)\,dx + \parint F(t,x,m) m\, dxdt
\\ & + \parint \left[  m\, H_p(t,x, m, Du)\cdot Du- m \,H(t,x,m,Du)\right]dxdt,
\\
  \into  m_0\, \tilde u(0)\, dx & =  \into G(x,\tilde m(T))\, \tilde m(T)\,dx + \parint F(t,x,\tilde m)\tilde m\, dxdt
\\ & + \parint \left[ \tilde m\, H_p(t,x,\tilde m, D\tilde u)\cdot D\tilde u- m \,H(t,x,\tilde m,D\tilde u)\right]dxdt,
 \end{split}
    \end{displaymath}
and therefore
\begin{equation}\label{eq:42}
\begin{split}
 0 \ge &\ds \int_{\Omega} (  G(m(T))-G(\tilde m(T)))\,   (m(T)-\tilde m(T))\,dx
 \ds + \int_0^T\int_{\Omega}( F(t,x,m)    - F(t,x,\tilde m))   (m-\tilde m) \, dxdt\\
& \ds + \int_0^T\int_{\Omega}  \tilde m \left[   H(t,x,m,Du) - H(t,x,\tilde m,D\tilde u)- H_p(t,x,  \tilde m, D  \tilde u)\cdot (Du-D\tilde u)  \right]
  dxdt\\
& \ds + \int_0^T\int_{\Omega}    m \left[   H(t,x,\tilde m,D\tilde u) - H(t,x, m,Du)- H_p(t,x,   m, D u)\cdot (D\tilde u-D u)  \right]
  dxdt.
\end{split}
\end{equation}
Let us define
\begin{displaymath}
  \begin{split}
 E(t,x, m_1,p_1,m_2, p_2) = &- (H( t,x, m_1,p_1) - H(t,x, m_2,p_2) ) (m_1-m_2)  \\ &+ \left( m_1  H_p( t,x, m_1,p_1) - m_2  H_p(t, x, m_2,p_2)\right) (p_1-p_2)\\ & + ( F(t,x,m_1)    - F(t,x, m_2))   (m_1- m_2)
\end{split}
\end{displaymath}
and  $  r= p_2-p_1$, $p_s= p_1+sr$,  $z=m_2-m_1$, $m_s= m_1+sz$.
From the assumptions, the function $h: s\mapsto -z H(t,x,m_s, p_s) + m_sH_{p}(t,x,m_s, p_s) \cdot r  
+ zF(t,x,m_s)$
is increasing on $[0,1]$; this yields that $ E(t,x,m_1,p_1,m_2, p_2) \ge 0$ for all $m_1, m_2 \ge 0$
and all $p_1,p_2\in \R^N$  and that $ E(t,x,m_1,p_1,m_2, p_2) = 0$ if and only if $m_1=m_2$ and
 $p_1=p_2$ . Thus, (\ref{eq:42}) and the assumptions imply
\begin{equation}\label{eq:43}
\begin{split}
\ds \int_{\Omega} (  G(x,m(T))-G(x,\tilde m(T)))\,   (m(T)-\tilde m(T))\,dx&=0,\\
 \ds \int_0^T\int_{\Omega}  E(t,x,m,Du,\tilde m, D \tilde u) \, dxdt&=0.
\end{split}
\end{equation}
From (\ref{eq:43}), we deduce that $m=\tilde m$ and $Du=D\tilde u$ a.e. in $Q_T$,
and that $G(\cdot, m(T,\cdot))=G(\cdot,\tilde m (T, \cdot))$ a.e. in $\Omega$. We then deduce that, in  a weak sense,  $(u-\tilde u)_t=0$ and $u(T)=\tilde u(T)$, hence $u=\tilde u$ a.e. in $Q_T$.
\end{proof}


\section{Existence and uniqueness for singular congestion}
\label{sec:exist-uniq-sing}
We consider the limit case when $\mu=0$; we aim at proving the existence and uniqueness of a weak solution of (\ref{MFGu})-(\ref{MFGbc})
as defined in Definition~\ref{defsing}.
Recall that it is not restrictive to  assume (\ref{eq:41}).
To prove  existence, we consider the weak solutions $(u^\mu, m^\mu)$ of 
\begin{eqnarray}
  \label{MFGumu} -\partial_t u ^\mu - \nu \Delta u^\mu + H(t,x, \mu+m^\mu,Du^\mu)  = F(t,x,m^\mu)\,,
  & \quad (t,x)\in (0,T)\times \Omega
  \\   \label{MFGmmu}
  \partial_t m^\mu  -\nu \Delta m^\mu -{\rm div} (m^\mu H_p(t,x,\mu+m^\mu,Du^\mu))=0\,, &\quad  (t,x)\in (0,T)\times \Omega
  \\   \label{MFGbcmu}
  m^\mu(0,x)=m_0(x)\,,\,\,u^\mu (T,x)= G(x,m^\mu(T))\,, & \quad x\in \Omega\,  
\end{eqnarray}
 for $\mu>0$ tending to $0$. By \rife{eq:18}, \rife{eq:20}, we already know that the following estimates hold
\begin{eqnarray}
\label{eq:33}
\into G(x,m^\mu(T))m^\mu(T) dx +  \parint F(t,x,m^\mu)m^\mu\, dxdt  
 + \| (m^\mu)^{\frac\alpha{\beta-1}+1}\|_{\parelle{\frac{N+2}N}}
\leq C,
\\
\label{eq:36}
 \parint m^\mu \left\{H_p(t,x, \mu+ m^\mu,Du^\mu)\cdot Du^\mu- H(t,x, \mu+m^\mu,Du^\mu)\right\}\,dxdt \leq C,
\\\label{eq:34}
\parint \frac{|Du^\mu|^\beta}{(  m^\mu+\mu)^\alpha} \, dxdt +
\parint m^\mu\, \frac{|Du^\mu|^\beta}{ (m^\mu+\mu)^\alpha} \, dxdt   \leq C,
\end{eqnarray}
for some $C$ independent of $\mu$.
\vskip0.4em
With the same proof as for   Lemma~\ref{sec:non-sing-cong-2}, we obtain:
\begin{lemma}
\label{sec:exist-uniq-sing-3}
For each $0<\mu<1$, the following functions defined on $Q_T$,\\ 
$w^\mu=    \mathds{1}_{\{m^\mu +\mu\le 1\}}   m^\mu  \frac {|Du^\mu|^{\beta-1}}{(\mu+m^\mu)^\alpha}$ and  $z^\mu=    \mathds{1}_{\{m^\mu +\mu> 1\}}   \sqrt{m^\mu}  \frac {|Du^\mu|^{\beta-1}}{(\mu+m^\mu)^\alpha} $,   are such that 
\[\left|m^\mu H_p(t,x,\mu+m^\mu, Du^\mu)\right| \le c_2 (m^\mu + w^\mu + \sqrt {m ^\mu } z^\mu),\] 
and the family $(w^\mu)$ is bounded in $L^{\beta '} (Q_T)$, the family $z^\mu$ is bounded in $\parelle2$ and also relatively compact if $\beta<2$.
\end{lemma}

\begin{remark}
  \label{sec:exist-uniq-sing-5}
In fact, the estimate of Lemma~\ref{sec:exist-uniq-sing-3} shows that $m^\mu H_p(t,x,\mu+m^\mu, Du^\mu)$ is actually bounded in $L^{1+\epsilon}(Q_T)$ for some $\epsilon>0$. This is obvious for $m^\mu$ from estimate \rife{eq:33} and so for $w^\mu$, which is bounded   in $\parelle{\beta'}$. 
As for the term  $\sqrt {m ^\mu } z^\mu$, we see that, for $0<\epsilon<1$
$$
\parint (\sqrt {m ^\mu } z^\mu)^{1+\epsilon} dxdt  \leq \left(\parint (z^\mu)^{2} dxdt\right)^{\frac{1+\epsilon}2} \left(\parint (m^\mu)^{\frac{1+\epsilon}{1-\epsilon}} dxdt\right)^{\frac{1-\epsilon}2} 
$$
so using the bound of $z^\mu$ in $\parelle2$ and the estimate \rife{eq:33} for $m^\mu$, the right-hand side is bounded as soon as $\epsilon$ is sufficiently small.
\end{remark}

\begin{lemma}
  \label{sec:exist-uniq-sing-4}
 \begin{enumerate}
   \item    The family $(m^\mu)$ is relatively compact in $L^1(Q_T)$  and in $C( [0,T]; W^{-1, r}(\Omega) )$ for some  $r>1$
   
  \item There exist $m\in L^1(Q_T)$ and $u \in L^q(0,T; W^{1,q}(\Omega))$ for any $q<\frac{N+2}{N+1}$ such that, after the extraction of a subsequence (not relabeled), $m^\mu\to m$ in $ L^1(Q_T)$ and almost everywhere, $u^\mu \to u$ and $Du^\mu\to Du$ in $\elle1$ and almost everywhere. Moreover,
 \begin{eqnarray}
\label{eq:39}
  \parint m^{-\alpha} |Du|^\beta   \mathds{1}_{\{m>0\}}\, dxdt +
\parint m^{1-\alpha}\, |Du|^\beta \mathds{1}_{\{m>0\}} \, dxdt   < +\infty,\\
\label{eq:40}
\parint F(x,m)m\,dxdt+ 
\| m^{\frac\alpha{\beta-1}+1}\|_{\parelle{\frac{N+2}N}}<\infty,
\end{eqnarray}
and $Du=0$ a.e. in $\{m=0\}$.

\item  If $\beta\le 2$ and 
  $\alpha< 1$, then  $m^\mu H_p(t,x,\mu +m^\mu, Du^\mu) \to   m H_p(t,x,m, Du) $ in $L^1(Q_T)$.
  If $\beta<2 $ and $1\le \alpha\le   \frac {4(\beta-1)}\beta$ or  $\beta=2 $ and $1\le \alpha<2$,
 then $m^\mu H_p(t,x,\mu +m^\mu, Du^\mu) \rightharpoonup  m H_p(t,x,m, Du) \mathds{1}_{\{m>0\}}$ weakly in $L^1(Q_T)$.
 
 \item $m\in \continue 1$ and $m^\mu(t)$ converges to $m(t)$ weakly in $\elle1$ for any $t\in [0,T]$.  If $\beta<2$, then $m^\mu\to m$ in $C([0,T]; L^1(\Omega))$. Finally, (\ref{eq:37}) holds for any $\varphi\in C^\infty_c([0,T)\times \Omega)$.
\end{enumerate}
\end{lemma}

\begin{proof}
The compactness of $m^\mu, u^\mu$ and $Du^\mu$ are established exactly as in Proposition \ref{sec:non-sing-cong-5}. 
Then we obtain  (\ref{eq:39}) by using (\ref{eq:34}) and Fatou's Lemma in the set $\{m>0\}$. We obtain (\ref{eq:40}) as an easy  consequence of (\ref{eq:33}). If $Du$ did not vanish a.e. in $\{m=0\}$, there would exist $\rho>0$ and a measurable subset $E$ of $\{m=0\}$  in which $|Du|>\rho$. Then, by Fatou's lemma, $\int_E(m^\mu +\mu)^{-\alpha} |Du^\mu|^\beta dxdt $ would tend to $+\infty$, in contradiction with (\ref{eq:34}).  Therefore $Du=0$ a.e. in $\{m=0\}$. 
\\
To pass to the limit in $m^\mu H_p(t,x,\mu +m^\mu, Du^\mu) $, we first observe that this sequence  is equi-integrable in $\parelle1$ as a consequence of Remark \ref{sec:exist-uniq-sing-5}. Let us now consider the case when $\alpha<1$:  from (\ref{grow}) with $\mu=0$, 
 $(t,x,m, p)\mapsto m H_p (t,x,m,p)$ is well defined in $\overline Q_T\times [0,+\infty)\times \R^N$, and 
$m^\mu H_p(t,x,\mu+m^\mu, Du^\mu)\to m H_p(t,x,m, Du)$ almost everywhere.  Then by Vitali's theorem,
 $m^\mu H_p(t,x,\mu +m^\mu, Du^\mu) \to   m H_p(t,x,m, Du) $ in $L^1(Q_T)$.
\\
 If $\beta<2 $ and $1\le \alpha\le   \frac {4(\beta-1)}\beta$ or  $\beta=2 $ and $1\le \alpha<2$, 
then naming $\xi^\mu= m^\mu H_p(t,x,\mu+m^\mu, Du^\mu)$ for brevity,  we proceed   through the following steps:
\begin{enumerate}
\item By Remark \ref{sec:exist-uniq-sing-5}, we can extract a subsequence, not relabeled, such that $\xi^\mu\rightharpoonup \xi$ in 
$L^{1+\epsilon} (Q_T)$ weak.
\item  For any bounded and smooth function $\phi$, Lebesgue theorem implies that  $\phi(m^\mu)\to \phi(m)$ in $L^p(Q_T)$ for any $p\ge 1$. Therefore, $\xi^\mu   \phi(m^\mu)    \rightharpoonup \xi \phi(m)$ in $L^{1} (Q_T)$ weak.
\item If $\phi$ is also supported in $[\delta,+\infty)$ for some positive $\delta$, then using the almost everywhere convergence and Vitali's theorem,   $\xi^\mu   \phi(m^\mu)  \to  m H_p(t,x,m, Du)  \phi(m)$ in $L^1(Q_T)$. 
 This implies that $(\xi - m H_p(t,x,m, Du)) \phi(m)=0$, and that $\xi$ coincides with $m H_p(t,x,m, Du)$ at almost every $(t,x)$ such that $m(t,x)>0$.
\item We notice  that, by definition of $\xi^\mu$ and assumption \rife{grow},  
\[
  \begin{split}
\ds \parint |\xi^\mu \phi(m^\mu) |& dxdt  \le \ds  c_2  \parint    m^\mu| \phi(m^\mu) | dxdt 
\\ &\ds   +  c_2  \left(\parint     (m^\mu) ^{\beta- \alpha} | \phi(m^\mu) | dxdt \right) ^{\frac  1 \beta}
\left(\parint     \frac {|Du^\mu| ^\beta}  { (\mu+ m^\mu) ^{\alpha}}  | \phi(m^\mu) |    dxdt \right) ^{\frac  1 {\beta'}}  
  \end{split}
\]
and  (\ref{eq:34}) implies  that
\begin{displaymath}
\ds \parint |\xi^\mu \phi(m^\mu) | dxdt  \le  \ds  c_2  \parint    m^\mu| \phi(m^\mu) | dxdt +C  \left(\parint     (m^\mu) ^{\beta- \alpha} | \phi(m^\mu) | dxdt \right) ^{\frac  1 \beta}.    
\end{displaymath}
Then using the weak convergence of  $\xi^\mu \phi(m^\mu) $ in $\parelle1$ and the fact that $0<\beta-\alpha\le 1$, we can deduce
\begin{displaymath}
  \begin{split}
   \parint |\xi \phi(m) | dxdt  \le &\ds \liminf_{\mu\to 0} \left( c_2  \parint    m^\mu| \phi(m^\mu) | dxdt +C  \left(\parint     (m^\mu) ^{\beta- \alpha} | \phi(m^\mu) | dxdt \right) ^{\frac  1 \beta}    \right)\\
=&     c_2  \parint    m| \phi(m) | dxdt +C  \left(\parint     m ^{\beta- \alpha} | \phi(m) | dxdt \right) ^{\frac  1 \beta}    .
  \end{split}
\end{displaymath}
Taking now $\phi(m)= \exp(-Km)$, and letting $K$ tend to $+\infty$ yields that $\int_{\{m=0\}} |\xi|dx dt =0$, and that $\xi=0$ a.e. in $\{m=0\}$. Notice that we used $\alpha<\beta$ in this step, in particular if $\beta=2$ we needed the restriction $\alpha<2$.
\item Collecting the above results, we obtain that $\xi=m H_p(t,x,m, Du) \mathds{1}_{\{m>0\}}$.
\end{enumerate}
Finally, the weak convergence of $m^\mu(t)$ in $\parelle1$ can be justified as in Proposition   \ref{sec:non-sing-cong-5} (as well as the strong $\continue1$ convergence if $\beta<2$), and obtaining equation (\ref{eq:37}) is a consequence of the previous points.
 \end{proof}
 
 We conclude the existence part by showing that $u$ satisfies (\ref{eq:38}) and the energy equality holds true.
 
  \begin{prop}
    \label{sec:exist-uniq-sing-6}
Assume that  $\beta<2 $ and $0< \alpha\le   \frac {4(\beta-1)}\beta$ or  $\beta=2 $ and $0<\alpha<2$.
Let $(u,m)$ be given by  Lemma   \ref{sec:exist-uniq-sing-4}. Then  inequality (\ref{eq:38}) holds true for every nonnegative $\vfi\in C^\infty_c((0,T]\times \Omega)$. \\
Moreover, $(u,m)$ satisfies the energy identity (\ref{en-id}).
  \end{prop}
  \begin{proof}
  From (\ref{eq:17}) and (\ref{eq:41}), we deduce that  for every nonnegative $\vfi\in C^\infty_c((0,T]\times \Omega)$,
 \begin{equation}
 \label{prefat}
   \begin{split}
 \int_0^T \into u^\mu\, \vfi_t\, dxdt & - \nu \int_0^T \into u^\mu\,\Delta\vfi\, dxdt 
 + \int_0^T\into H(t,x,\mu+m^\mu,Du^\mu)  \vfi    \mathds{1}_{\{m>0\}}\, dxdt
 \\
 & \le \int_0^T \into F(t,x,m^\mu)\vfi\, dxdt + \into G(x,m^\mu(T))\vfi(T)\, dx.     
   \end{split}
 \end{equation}
 By using the a.e. convergence of $H(t,x,\mu+m^\mu,Du^\mu)     \mathds{1}_{\{m>0\}}$ to 
 $H(t,x,m,Du)     \mathds{1}_{\{m>0\}}$ and Fatou's lemma, 
 \begin{displaymath}
  \parint   H(t,x, m,Du) \varphi(t,x) \mathds{1}_{\{m>0\}} dxdt \le \liminf_{\mu\to 0} \parint   H(t,x, \mu+ m^\mu,Du^\mu)  \mathds{1}_{\{m>0\}} 
 \varphi(t,x) dxdt. 
 \end{displaymath}
By properties of Young measures, and since $G(x,m^\mu(T))$ is equi-integrable, 
$$
\lim_{\mu\to 0} \into G(x,m^\mu(T))\vfi(T)\, dx = \into \int_{\R} G(x,\lambda)d\nu_x(\lambda)\vfi(T)\, dx \,.
$$
We easily pass to the limit in the other terms  of \rife{prefat} since $u^\mu$ and $F(t,x,m^\mu)$ converge  in $\elle1$. So we end up with
\be
\label{postfat}
   \begin{split}
 \int_0^T \into u \, \vfi_t\, dxdt & - \nu \int_0^T \into u \,\Delta\vfi\, dxdt 
 + \int_0^T\into H(t,x, m ,Du )  \vfi    \mathds{1}_{\{m>0\}}\, dxdt
 \\
 & \le \int_0^T \into F(t,x,m )\vfi\, dxdt + \into \int_{\R} G(x,\lambda)d\nu_x(\lambda)\vfi(T)\, dx\,.     
   \end{split}
 \end{equation}
Thanks to \rife{postfat}, the conclusion of Lemma~\ref{sec:non-sing-cong-8} is also true. Therefore,  
$$
\lim\limits_{\mu\to 0} \into u^\mu(0)m_0\, dx \leq \langle u(0), m_0\rangle\,.
$$
Starting now from the energy identity for the solution of (\ref{MFGu})-(\ref{MFGbc}) with $\mu>0$:
\begin{align*}& 
 \parint m^\mu \left( H_p(t,x, \mu+m^\mu,Du^\mu)\cdot Du^\mu -H(t,x, \mu+m^\mu,Du^\mu)\right) \,dxdt
\\
&  +\into m^\mu(T) G (x,m^\mu(T)) dx+\parint m^\mu F (t,x,m^\mu)\, dxdt
\quad = \into u^\mu(0)m_0\, dx  
\end{align*}
using Fatou's lemma thanks to (\ref{eq:2}) and (\ref{F0}), we obtain
\begin{align*}
 \limsup_{\mu\to 0} \into  G (x,m^\mu(T))m^\mu(T) dx & \leq \langle u(0), m_0\rangle-\parint mF(t,x,m)\, dxdt
 \\ & -   \parint \mathds{1}_{\{m>0\}} m \left( H_p(t,x, m,Du)\cdot Du -H(t,x, m,Du)\right) \,dxdt.
\end{align*}
Reasoning exactly as in the previous section, from the estimate \rife{eq:33} and assumption \rife{G2} we obtain both $G(x,m(T))m(T)\in \elle1$ and $\int_{\R} G(x,\lambda)\lambda\,d\nu_x(\lambda)\in \elle1$. Therefore, we are allowed to use (\ref{ineq-s}) in Lemma~\ref{compare-s} with $\tilde m=m$ and $\tilde u=u$, and we get
\begin{align*}
 & \limsup_{\mu\to 0} \into  G(x,m^\mu(T))m^\mu(T) dx   \leq \langle u(0), m_0\rangle-\parint mF(t,x,m)\, dxdt
 \\ & \qquad \qquad -   \parint \mathds{1}_{\{m>0\}} m \left( H_p(t,x, m,Du)\cdot Du -H(t,x, m,Du)\right) \,dxdt 
 \\
&  \qquad \leq \into \int_{\R} G(x,\lambda)d\nu_x(\lambda)\,m(T) dx.
\end{align*}
Now we proceed as in Lemma \ref{sec:exist-uniq-non-2} using the monotonicity of $G$ and we conclude that $G(x,m^\mu(T))\to G(x,m(T))$ in $\elle1$ and the energy identity holds true. Coming back to \rife{postfat}, now we know that $\int_{\R} G(x,\lambda)d\nu_x(\lambda)= G(x,m(T))$ and so \rife{eq:38} is proven.
  \end{proof}
  
 We conclude with the uniqueness of the weak solution. To this purpose, since $u$ is now only a subsolution, we will need a stronger version of the energy equality for the uniqueness argument to succeed.

 \begin{lemma}\label{newlem} Let $(u,m)$ be a weak solution of (\ref{MFGu})-(\ref{MFGbc}). Then 
  $u m\in \parelle1$ and 
\be\label{idt}\begin{split}
 \into  m(t)\, u(t)\, dx & =  \into G(x,m(T))\, m(T)\,dx + \int_t^T \!\! \into F(s,x,m) m\, dxds
\\ & + \int_t^T \!\! \into \left[  m\, H_p(s,x, m, Du)\cdot Du- m \,H(s,x,m,Du)\right] \, \mathds{1}_{\{m>0\}}dxds
\end{split}
\ee
 for almost every $t\in (0,T)$.
 \end{lemma}
 
 \begin{proof}
Let us set $u_k:= (u-k)^+$. We first observe that, whenever $u$ is  a subsolution, then $u_k$ is also  a subsolution of a  similar problem, namely
\begin{align*}
\int_0^T \into u_k\, \vfi_t\, dxdt & - \nu \int_0^T \into u_k\,\Delta\vfi\, dxdt 
+ \int_0^T\into H(t,x,m,Du)\mathds{1}_{\{m>0\}}\,  \mathds{1}_{\{u>k\}}\,\vfi\, dxdt
\\
& \leq  \int_0^T \into F(t,x,m)\mathds{1}_{\{u>k\}}\,\vfi\, dxdt + \into   (G(x,m(T))-k)^+ \,\vfi(T)\, dx
\end{align*}
Henceforth, we can proceed exactly as in Lemma \ref{compare-s} with $\tilde m=m$ and we obtain (recall that $\md(t,x)= (m(t) \star \rho_\de)(x)$)
\begin{equation}\label{gkumd}
\begin{split} & 
\int_t^T \int _\Omega  [\mathds{1}_{\{m>0\}} H(s,x,m,Du)\md - Du \cdot w_\de ] \,\mathds{1}_{\{u>k\}}\, dxds + \into  \md(t)\, u_k(t) \,dx
\\
& \quad \leq   \int_t^T\into F(s,x,m)\md\, \mathds{1}_{\{u>k\}}\, dxds + \into    (G(x,m(T))-k)^+\,  \md(T) \,dx 
\end{split}
\end{equation}
which holds for almost every $t$ since $\md \in  \parelle\infty$ (it is indeed even continuous) and $u_k\in \parelle1$ and is bounded below. Starting from \rife{gkumd} and following the same steps as in Lemma \ref{compare-s},  we obtain that
$$
\into  \md(t)\, u_k(t) \,dx \leq   \int_t^T\into \phi_\de\, \mathds{1}_{\{u>k\}}\, dxds + \into    \chi_\de\, \mathds{1}_{\{G(x,m(T))>k\}} \,dx 
$$
for some sequences $\phi_\de$, $\chi_\de$ which are convergent, and therefore equi-integrable, in $\parelle1$ and in $\elle1$ respectively. This implies that
$$
\sup_\de \into  \md(t)\, u_k(t) \,dx \mathop{\to}^{k\to \infty} 0\,.
$$
Since $u-u_k$ is bounded, we know that $\md(u-u_k)$ strongly converges in $\parelle1$ as $\de \to 0$. Therefore, a  standard argument allows us to conclude that 
\be\label{uml1}
\md\, u\to m  u \qquad \hbox{in $\parelle1$.}
\ee
This means that we can repeat the argument of Lemma \ref{compare-s} integrating either in $(0,t)$ or in $(t,T)$ obtaining both
\begin{align*} & 
 \int_0^t\into [H(s,x,m,Du)\md - Du \cdot w_\de ] \mathds{1}_{\{m>0\}}\, dxds + \langle\md(0)\, u(0)\rangle \,dx
\\
& \quad \leq   \int_0^t \into F(s,x,m)\md\,   dxds + \into    u(t) \,  \md(t) \,dx 
\end{align*}
and
\begin{align*} & 
 \int_t^T\into [H(s,x,m,Du)\md - Du \cdot w_\de ]\mathds{1}_{\{m>0\}} \, dxds + \into  \md(t)\, u(t) \,dx
\\
& \quad \leq   \int_t^T\into F(s,x,m)\md \, dxds + \into    G(x,m(T)) \,  \md(T) \,dx 
\end{align*}
for almost every $t\in (0,T)$. Passing to the limit as $\de\to 0$ is now allowed thanks to \rife{uml1}, besides the arguments already given in Lemma \ref{compare-s}. Finally, as $\de \to 0$ we obtain  both
\begin{align*}
\langle m_0\,, u(0)\rangle \, dx & \leq  \into u(t)\, m(t)\,dx + \int_0^t \into  F(s,x,m) m\, dxds
\\ & + \int_0^t \into \left[ m\, H_p(s,x, m, D  u)\cdot Du-  m \,H(s,x,m,Du)\right]\, \mathds{1}_{\{m>0\}}dxds
\end{align*}
and
\begin{align*}
\into   m(t)\, u(t)\, dx & \leq  \into  G(x,m(T))\, m(T)\,dx + \int_t^T \into F(s,x,m) m\, dxds
\\ & + \int_t^T \into  \left[   m\, H_p(s,x,  m, D  u)\cdot Du-   m \,H(s,x,m,Du)\right]\, \mathds{1}_{\{m>0\}}dxds\,.
\end{align*}
On account of the energy equality \rife{en-id} which holds by definition of weak solution, we conclude that \rife{idt} holds true, for almost every $t\in (0,T)$.
\end{proof}

 \begin{lemma}
\label{sec:bf-first-draft}
Under all the assumptions made in Theorem~\ref{sec:case-sing-cong-1}, there is a unique weak solution of (\ref{MFGu})-(\ref{MFGbc}).
  \end{lemma}
 \begin{proof}
   Let $(u,m)$ and $(\tilde u, \tilde m)$ be two weak solutions of (\ref{MFGu})-(\ref{MFGbc}). From Lemma~\ref{compare-s} and Proposition \ref{sec:exist-uniq-sing-6}, we know that 
    \begin{displaymath}
      \begin{split}
        \into  m_0\, u(0)\, dx & \leq  \into G(x,m(T))\, \tilde m(T)\,dx + \parint F(t,x,m)\tilde m\, dxdt
\\ & + \parint \left[ \tilde m\, H_p(t,x,\tilde m, D\tilde u)\cdot Du- \tilde m \,H(t,x,m,Du)\right]\, \mathds{1}_{\{m>0\,,\, \tm>0\}}dxdt,
\\
        \into \ m_0\, \tilde u(0)\, dx & \leq  \into G(x,\tilde m(T))\,  m(T)\,dx + \parint F(t,x,\tilde m) m\, dxdt
\\ & + \parint \left[  m\, H_p(t,x, m, D u)\cdot D\tilde u-  m \,H(t,x,\tilde m,D\tilde u)\right]\, \mathds{1}_{\{m>0\,,\, \tm>0\}} dxdt,
\\
  \into  m_0\, u(0)\, dx & =  \into G(x,m(T))\, m(T)\,dx + \parint F(t,x,m) m\, dxdt
\\ & + \parint \left[  m\, H_p(t,x, m, Du)\cdot Du- m \,H(t,x,m,Du)\right] \, \mathds{1}_{\{m>0\}}dxdt,
\\
  \into  m_0\, \tilde u(0)\, dx & =  \into G(x,\tilde m(T))\, \tilde m(T)\,dx + \parint F(t,x,\tilde m)\tilde m\, dxdt
\\ & + \parint \left[ \tilde m\, H_p(t,x,\tilde m, D\tilde u)\cdot D\tilde u- m \,H(t,x,\tilde m,D\tilde u)\right]\, \mathds{1}_{\{\tm>0\}}dxdt,
 \end{split}
    \end{displaymath}
and therefore
\begin{equation}\label{eq:45}
\begin{split}
 0 \ge &\ds \int_{\Omega} (  G(x,m(T))-G(x,\tilde m(T)))\,   (m(T)-\tilde m(T))\,dx
 \ds + \int_0^T\int_{\Omega}( F(t,x,m)    - F(t,x,\tilde m))   (m-\tilde m) \, dxdt\\
& \ds + \int_0^T\int_{\Omega} \mathds{1}_{\{m>0\,,\, \tm>0\}}\tilde m \left[   H(t,x,m,Du) - H(t,x,\tilde m,D\tilde u)- H_p(t,x,  \tilde m, D  \tilde u)\cdot (Du-D\tilde u)  \right]
  dxdt\\
& \ds + \int_0^T\int_{\Omega}  \mathds{1}_{\{m>0\,,\, \tm>0\}} m \left[   H(t,x,\tilde m,D\tilde u) - H(t,x, m,Du)- H_p(t,x,   m, D u)\cdot (D\tilde u-D u)  \right]
  dxdt\\
& \ds + \int_0^T\int_{\Omega}  \mathds{1}_{\{m>0\,,\, \tm=0\}} m \left[   - H(t,x, m,Du)+ H_p(t,x,   m, D u)\cdot D u  \right]
  dxdt\\
& \ds + \int_0^T\int_{\Omega}  \mathds{1}_{\{\tm>0\,,\, m=0\}} \tm \left[   -H(t,x,\tilde m,D\tilde u) + H_p(t,x,   \tm, D \tilde u)\cdot D\tilde u \right]
  dxdt.
\end{split}
\end{equation}
Note that  from (\ref{eq:1}) and the convexity of $H(t,x,m,p)$ w.r.t $p$, 
\begin{eqnarray}
\label{eq:47}
\tilde m>0 \quad &\Rightarrow  \quad  - H(t,x,\tilde m,D\tilde u)+ H_p(t,x,  \tilde m, D  \tilde u)\cdot D\tilde u \ge - H(t,x,\tilde m,0) \ge 0,\\
\label{eq:48} m>0 \quad &\Rightarrow  \quad - H(t,x,m,D u)+ H_p(t,x,   m, D  u)\cdot D u\ge - H(t,x, m,0) \ge 0.
\end{eqnarray}
Therefore all integrals in \rife{eq:45} involve nonnegative functions, which then  must be zero almost everywhere. We deduce that 
\begin{equation}\label{eq:46}
\begin{split}
\ds  \int_{\Omega} (  G(m(T))-G(\tilde m(T)))\,   (m(T)-\tilde m(T))\,dx&=0,\\
 \ds  \int_0^T\int_{\Omega} \mathds{1}_{\{m>0\,,\, \tm>0\}}  E(t,x,m,Du,\tilde m, D\tilde u) \, dxdt=0,\\
  \ds  \int_0^T\int_{\Omega}  \mathds{1}_{\{m=0, \tilde m>0\}}  \tilde m 
 ( - H(t,x,\tilde m,D\tilde u)+ H_p(t,x,  \tilde m, D  \tilde u)\cdot D\tilde u  +F(t,x,\tilde m )-F(t,x,0)  ) dxdt
&=0,\\
 \ds  \int_0^T\int_{\Omega}  \mathds{1}_{\{\tilde m=0,m>0\}} m  ( - H(t,x,m,D u)+ H_p(t,x,   m, D  u)\cdot D u+F(t,x,m)-F(t,x,0 ))&=0,
\end{split}
\end{equation}
where $E$ has been defined in the proof of Lemma~\ref{sec:exist-uniq-non-3}. From the last two equations in (\ref{eq:46}), (\ref{eq:47})-(\ref{eq:48}) and (\ref{eq:49}), we deduce that the sets
 $\{\tilde m=0,m>0\}$ and $\{ m=0,\tilde m>0\}$ have measure $0$.
From the second equation in (\ref{eq:46}) and the assumptions made for uniqueness (which apply to  the function $E$), as in Lemma~\ref{sec:exist-uniq-non-3} we deduce that $m=\tilde m$ and $Du=D\tilde u$ a.e. in $\{m>0,\tilde m>0\}$.
Combining all the previous observations,  $m=\tilde m$ almost everywhere in $Q_T$. Using also that  that $Du =0$ a.e. in $\{m=0\}$ and $D\tilde u =0$ a.e. in $\{\tilde m=0\}$,
we deduce that $Du=D\tilde u$ a.e.  in $Q_T$. 
Therefore,  $u-\tilde u = \into (u(t)-\tilde u(t))dx$. By applying Lemma \ref{newlem} to $u$ and $\tilde u$, subtracting the two equalities on account of the above informations we finally get
$$
\into m(t) (u-\tilde u)(t)\, dx =0 \,.
$$
Since $u-\tilde u $ is only time-dependent, this  implies  $u=\tilde u$ a.e. in $Q_T$.   
  \end{proof}

\bigskip 
\noindent{\bf Acknowledgement.} \quad  
Most of this work was done during a visit of A. Porretta at the University Paris Diderot (Paris VII). A. Porretta wishes to thank this institution for the invitation and the very kind hospitality given on this occasion. The work was also  supported by the Indam  Gnampa project 2015 {\it Processi di diffusione degeneri o singolari legati al controllo di dinamiche stocastiche} and by  ANR projects ANR-12-MONU-0013 and ANR-16-CE40-0015-01.

\end{document}